\documentclass[12pt,a4paper]{amsart}
\usepackage{amssymb}
\usepackage[all,cmtip]{xy}

\calclayout

\newcommand{\+}{\protect\nobreakdash-}

\renewcommand{\:}{\colon}

\newcommand{\rarrow}{\longrightarrow}

\newcommand{\lrarrow}{\mskip.5\thinmuskip\relbar\joinrel\relbar\joinrel
 \rightarrow\mskip.5\thinmuskip\relax}

\DeclareMathOperator{\Ext}{Ext}
\DeclareMathOperator{\Hom}{Hom}
\DeclareMathOperator{\Gen}{Gen}
\DeclareMathOperator{\Rel}{Rel}

\DeclareMathOperator{\Ind}{\mathsf{Ind}}
\DeclareMathOperator{\Fun}{\mathsf{Fun}}
\DeclareMathOperator{\Fil}{\mathsf{Fil}}

\DeclareMathOperator{\rk}{rk}
\DeclareMathOperator{\im}{im}
\DeclareMathOperator{\coker}{coker}

\newcommand{\FP}{\mathrm{FP}}

\newcommand{\sA}{\mathsf A}
\newcommand{\sB}{\mathsf B}
\newcommand{\sC}{\mathsf C}
\newcommand{\sD}{\mathsf D}
\newcommand{\sE}{\mathsf E}
\newcommand{\sF}{\mathsf F}

\newcommand{\sJ}{\mathsf J}
\newcommand{\sK}{\mathsf K}
\newcommand{\sL}{\mathsf L}
\newcommand{\sM}{\mathsf M}
\newcommand{\sN}{\mathsf N}
\newcommand{\sP}{\mathsf P}
\newcommand{\sS}{\mathsf S}
\newcommand{\sT}{\mathsf T}
\newcommand{\sU}{\mathsf U}

\newcommand{\Sets}{\mathsf{Sets}}
\newcommand{\Ab}{\mathsf{Ab}}

\newcommand{\cS}{\mathcal S}

\newcommand{\bC}{\mathbf C}
\newcommand{\bM}{\mathbf M}

\newcommand{\fp}{\mathsf{fp}}

\newcommand{\sop}{\mathsf{op}}
\newcommand{\id}{\mathrm{id}}

\newcommand{\boZ}{\mathbb Z}

\newcommand{\modr}{{\operatorname{\mathsf{mod--}}}}
\newcommand{\Modr}{{\operatorname{\mathsf{Mod--}}}}
\newcommand{\Modrfp}{{\operatorname{\mathsf{Mod_{fp}--}}}}
\newcommand{\Modrfl}{{\operatorname{\mathsf{Mod_{fl}--}}}}
\newcommand{\Modrcot}{{\operatorname{\mathsf{Mod^{cot}--}}}}

\newcommand{\Section}[1]{\bigskip\section{#1}\medskip}
\theoremstyle{plain}
\newtheorem{thm}{Theorem}[section]

\newtheorem{lem}[thm]{Lemma}
\newtheorem{prop}[thm]{Proposition}
\newtheorem{cor}[thm]{Corollary}
\theoremstyle{definition}
\newtheorem{ex}[thm]{Example}
\newtheorem{exs}[thm]{Examples}
\newtheorem{rem}[thm]{Remark}
\newtheorem{rems}[thm]{Remarks}

\begin{document}

\author{Leonid Positselski}

\address{Institute of Mathematics, Czech Academy of Sciences \\
\v Zitn\'a~25, 115~67 Praha~1 \\ Czech Republic}

\email{positselski@math.cas.cz}

\title{Locally coherent exact categories}

\begin{abstract}
 A locally coherent exact category is a finitely accessible additive
category endowed with an exact structure in which the admissible
short exact sequences are the directed colimits of admissible
short exact sequences of finitely presentable objects.
 We show that any exact structure on a small idempotent-complete
additive category extends uniquely to a locally coherent exact structure
on the category of ind-objects; in particular, any finitely accessible
category has the unique maximal and the unique minimal locally coherent
exact category structures.
 All locally coherent exact categories are of Grothendieck type in
the sense of \v St\!'ov\'\i\v cek.
 We also discuss the canonical embedding of a small exact category into
the abelian category of additive sheaves in connection with the locally
coherent exact structure on the ind-objects, and deduce two periodicity
theorems as applications.
\end{abstract}

\maketitle

\tableofcontents

\section*{Introduction}
\medskip

 Noetherian rings and modules over them play an important role in
commutative algebra and representation theory.
 Over a Noetherian ring $R$, the category of finitely generated
modules $\modr R$ is abelian, and the abelian category of arbitrary
modules $\Modr R$ is locally Noetherian.
 Coherent rings are a natural generalization of Noetherian ones:
over a coherent ring, the category of finitely presentable modules
$\modr R$ is abelian, and the abelian category of arbitrary
modules $\Modr R$ is locally coherent.
 Still, there are many rings that are not coherent.

 The aim of this paper is to suggest a point of view allowing to
consider the category of modules over an arbitrary ring as ``locally
coherent'' in some sense.
 This comes at the cost of passing from abelian to exact categories
(in Quillen's sense).

 We explain that, for any ring $R$, the abelian category $\Modr R$
(as, more generally, any finitely accessible additive category) has
a complete lattice of  \emph{locally coherent exact structures}.
 The \emph{minimal} locally coherent exact structure is the \emph{pure}
exact structure.
 The \emph{maximal locally coherent exact structure} is the closest
locally coherent approximation to the abelian exact structure on
$\Modr R$.

 The maximal locally coherent exact structure keeps track of the short
exact sequences of finitely presentable $R$\+modules, but disregards
the short exact sequences in which the middle term and the cokernel
are finitely presentable while the kernel is not.
 All information about short exact sequences of the latter type is
destroyed by considering $\Modr R$ as an exact category with
the maximal (or any other) locally coherent exact structure.
 That is the price one has to pay for viewing the category of modules
over an arbitrary ring through coherent lens. 

 More generally, for any regular cardinal~$\kappa$ and any
$\kappa$\+accessible additive category $\sA$, an exact structure
on $\sA$ is said to be \emph{locally\/ $\kappa$\+coherent} if
the admissible short exact sequences in $\sA$ are precisely
the $\kappa$\+directed colimits of admissible short exact sequences
of $\kappa$\+presentable objects.
 We show that locally $\kappa$\+coherent exact structures on $\sA$
correspond bijectively to arbitrary exact structures on the full
subcategory $\sA_{<\kappa}$ of $\kappa$\+presentable objects in~$\sA$.
 (Notice that an abelian category $\sA$ is locally coherent if and
only if it is locally finitely presentable and its full subcategory of
finitely presentable objects is abelian; in other words, this means
that $\sA$ is the category of ind-objects of a small abelian category.)

 Beyond the abelian category $\sA=\Modr R$ of arbitrary $R$\+modules,
the additive category $\sA=\Modrfl R$ of flat $R$\+modules provides
an important example for our theory.
 There is a phenomenon of \emph{flat coherence}: for any regular
cardinal~$\kappa$, the kernel of any surjective morphism of
$\kappa$\+presentable flat $R$\+modules is again a $\kappa$\+presentable
flat $R$\+module~\cite[Lemma~4.1]{Pflcc}, \cite[Corollary~10.12]{Pacc}.
 In the language of the present paper, this observation appears as
a corollary of the fact that the category $\Modrfl R$ with its natural
exact structure is locally $\kappa$\+coherent for any ring~$R$.
 The exact structure on $\Modrfl R$ is pure, but it is not
$\kappa$\+pure, of course.
 
 So, for any small exact category $\sE$, the accessible category
$\Ind\sE$ of ind-objects of $\sE$ acquires the induced locally
coherent exact category structure.
 As a digression, we discuss the classical construction of
the canonical embedding of $\sE$ into the abelian category $\sK$
of left exact functors $\sE^\sop\rarrow\Ab$ in
Section~\ref{canonical-embedding-secn}.
 We observe that the category $\sK$ is locally finitely presentable
(and in fact, locally type $\FP_\infty$), and the locally coherent
exact category structure on $\Ind\sE$ is inherited from the abelian
exact structure on $\sK$ via a natural fully faithful embedding
$\Ind\sE\rarrow\sK$.

 The main results of this paper are two \emph{periodicity theorems}.
 Theorem~\ref{first-periodicity-theorem} generalizes
the flat/projective periodicity theorem of Benson--Goodearl
and Neeman~\cite{BG,Neem} and the fp\+projective periodicity theorem
of \v Saroch--\v St\!'ov\'\i\v cek~\cite{SarSt}, while
Theorem~\ref{second-periodicity-theorem} generalizes the cotorsion
periodicity theorem of Bazzoni, Cort\'es-Izurdiaga,
and Estrada~\cite{BCE}.
 The idea of these generalizations is to replace the split exact
category of finitely generated projective $R$\+modules (or
the abelian category of finitely presentable modules over
a coherent ring~$R$) by an arbitrary small exact category.

 As a particular case of Theorem~\ref{first-periodicity-theorem},
we deduce in Theorem~\ref{mlc-periodicity-theorem} a new version
of fp\+projective periodicity theorem applicable to an arbitrary
(and not necessarily coherent) ring~$R$.
 Another such version of fp\+projective periodicity was previously
obtained in the paper~\cite{BHP}.

 Beyond periodicity theorems, what can one do with locally
coherent exact categories?
 Let us explain the idea which motivated the present research.
 The results of~\cite[Theorem~9.39 and Corollary~9.42]{Pedg}, and
particularly~\cite[Theorem~8.19 and Corollary~8.20]{PS5}, describe
compact generators of the coderived categories of locally coherent
abelian DG\+categories, such as the category of curved DG\+modules
over a curved DG\+ring with a coherent underlying graded ring.
 The coherence condition involved is somewhat restrictive, as
mentioned in the beginning of this introduction.

 The definition of an \emph{exact DG\+category} was given and studied
in the paper~\cite{Pedg}; the \emph{abelian
DG\+categories}~\cite{Pedg,PS5} are a particular case.
 The definition of a locally coherent abelian DG\+category can be
found in~\cite[Section~9.5]{Pedg} and~\cite[Section~8.3]{PS5}.
 We hope that a suitable definition of a locally coherent exact
(rather than abelian) DG\+category can be worked out so that
the results concerning compact generators of coderived categories 
mentioned above would be generalizable to locally coherent exact
DG\+categories, such as the maximal locally coherent exact DG\+category
structure on the category of CDG\+modules over an arbitrary CDG\+ring.

\subsection*{Acknowledgement}
 I~am grateful to Silvana Bazzoni for a stimulating conversation.
 I~also want to thank Jan \v St\!'ov\'\i\v cek, Michal Hrbek, and
Jan Trlifaj for helpful discussions.
 Finally, I~wish to thank the anonymous referee for careful reading of
the manuscript and some helpful suggestions.
 The author is supported by the GA\v CR project 23-05148S and
the Czech Academy of Sciences (RVO~67985840).

\Section{Basic Properties of Locally $\kappa$-Coherent Exact Categories}
\label{basic-properties-secn}

 Let $\kappa$~be a regular cardinal.
 We refer to the beginning paragraphs of the Appendix for
the definitions, notation, and references concerning
$\kappa$\+presentable objects and $\kappa$\+accessible categories.

 Given an (additive) category $\sA$, consider the category $\bM^3(\sA)$
of composable pairs of morphisms in $\sA$, and its full subcategory
$\bC^3(\sA)\subset\bM^3(\sA)$ of three-term complexes in~$\sA$
(that is, composable pairs of morphisms with zero composition).
 We also consider the category $\bC^2(\sA)=\bM^2(\sA)$ of morphisms
in~$\sA$.

\begin{prop} \label{simple-diagrams-accessible}
 Let\/ $\sA$ be a $\kappa$\+accessible category.
 In this context: \par
\textup{(a)} The category of morphisms\/ $\bM^2(\sA)$ is
$\kappa$\+accessible.
 A morphism $A^1\rarrow A^2$ in\/ $\sA$ is $\kappa$\+presentable as
an object of\/ $\bM^2(\sA)$ if and only if both the objects $A^1$
and $A^2$ are $\kappa$\+presentable in\/~$\sA$. \par
\textup{(b)} The category of composable pairs of morphisms\/
$\bM^3(\sA)$ is $\kappa$\+accessible.
 A pair of morphisms $A^1\rarrow A^2\rarrow A^3$ in\/ $\sA$ is
$\kappa$\+presentable as an object of\/ $\bM^3(\sA)$ if and only if
all the three objects $A^1$, $A^2$, and $A^3$ are $\kappa$\+presentable
in\/~$\sA$. \par
\textup{(c)} Assume that the category\/ $\sA$ is additive.
 Then the category of three-term complexes\/ $\bC^3(\sA)$ is
$\kappa$\+accessible.
 A three-term complex $A^1\rarrow A^2\rarrow A^3$ in\/ $\sA$ is
$\kappa$\+presentable as an object of\/ $\bC^3(\sA)$ if and only if
all the three objects $A^1$, $A^2$, and $A^3$ are $\kappa$\+presentable
in\/~$\sA$.
\end{prop}

\begin{proof}
 Part~(a) is an easy particular case of
Proposition~\ref{nonadditive-functors-prop}, or alternatively,
a special case of Proposition~\ref{comma-accessible}.
 Part~(b) is also an easy particular case of
Proposition~\ref{nonadditive-functors-prop}; it can be also obtained
by an iterated application of Proposition~\ref{comma-accessible}.
 Part~(c) is an easy particular case of
Proposition~\ref{k-linear-functors-prop} (for $k=\boZ$).

 Alternatively, here is a simple direct proof of the ``if and only if''
assertion in part~(c), not depending on the assumption that $\sA$
is $\kappa$\+accessible, but only that $\kappa$\+directed colimits
exist in~$\sA$.
 Denoting by $\sA^3$ the Cartesian product of three copies of $\sA$,
we have to show that the forgetful functor $F^3\:\bC^3(\sA)\rarrow\sA^3$ 
preserves and reflects $\kappa$\+presentability of objects (cf.\
Lemma~\ref{product-accessible}).
 Indeed, the functor $F^3$ preserves $\kappa$\+presentability, since
its right adjoint functor $G^{3-}\:\sA^3\rarrow\bC^3(\sA)$, which
is easily constructed explicitly in terms of finite direct sums of
objects in $\sA$, preserves $\kappa$\+directed colimits
(the assumption that $\sA$ is additive is used here).
 The converse assertion follows from the fact that directed colimits
commute with finite limits in the category of sets, and holds in any
finite diagram category~\cite[Proposition~2.1]{Hen}.
\end{proof}

 We refer to the survey paper~\cite{Bueh} for the background material
on exact categories (in the sense of Quillen).
 The expositions in~\cite[Appendix~A]{Kel} and~\cite[Appendix~A]{Partin}
can be used as supplements.
 Notice that any $\kappa$\+accessible additive category $\sA$ and its
full subcategory of $\kappa$\+presentable objects $\sA_{<\kappa}$
are idempotent-complete by~\cite[Observation~2.4 and
Proposition~1.16]{AR}.
 Hence both $\sA$ and $\sA_{<\kappa}$ are weakly idempotent-complete
in the sense of~\cite[Section~7]{Bueh}.

 Let $\sA$ be a $\kappa$\+accessible additive category.
 We will say that an exact category structure on $\sA$ is \emph{locally
$\kappa$\+coherent} if the admissible short exact sequences in $\sA$
are precisely the $\kappa$\+directed colimits of admissible short
exact sequences in $\sA$ with the terms belonging to\/~$\sA_{<\kappa}$.
 The $\kappa$\+directed colimits here are taken in the category of
three-term complexes~$\bC^3(\sA)$.

\begin{lem} \label{short-exact-preserved-by-directed-colimits}
 In any locally $\kappa$\+coherent exact category, the class of all
(admissible) short exact sequences is preserved by
$\kappa$\+directed colimits.
\end{lem}

\begin{proof}
 Follows from Proposition~\ref{simple-diagrams-accessible}(c)
together with Proposition~\ref{accessible-subcategory}
(applied to the $\kappa$\+accessible category $\bC^3(\sA)$ and
the class $\sT$ of all short exact sequences in $\sA$ with
the terms belonging to~$\sA_{<\kappa}$).

 Notice that we have shown more than the lemma claims.
 Applied to the situation at hand,
Proposition~\ref{accessible-subcategory} tells us that the category
of short exact sequences in $\sA$ is $\kappa$\+accessible, and its
full subcategory of $\kappa$\+presentable objects consists of all
the short exact sequences of $\kappa$\+presentable objects in~$\sA$.
\end{proof}

\begin{rem} \label{accessible-monos-epis-as-dir-colims-remark}
 Let $\sA$ be a locally $\kappa$\+coherent exact category.
 Then it follows from
Lemma~\ref{short-exact-preserved-by-directed-colimits} that the classes
of admissible monomorphisms and admissible epimorphisms in $\sA$ are
also preserved by $\kappa$\+directed colimits.
 Therefore, the admissible monomorphisms in $\sA$ are precisely all
the $\kappa$\+directed colimits of admissible monomorphisms
with $\kappa$\+presentable domain and codomain.
 Similarly, the admissible epimorphisms in $\sA$ are precisely all
the $\kappa$\+directed colimits of admissible epimorphisms
with $\kappa$\+presentable domain and codomain.
\end{rem}

\begin{prop} \label{finitely-presentables-extension-closed}
 In any locally $\kappa$\+coherent exact structure on
a $\kappa$\+accessible additive category\/ $\sA$, the full subcategory
of $\kappa$\+presentable objects\/ $\sA_{<\kappa}$ is closed under
extensions in\/~$\sA$.
 So\/ $\sA_{<\kappa}$ inherits an exact category structure from\/~$\sA$.
\end{prop}

\begin{proof}
 Let $0\rarrow A^1\rarrow A^2\rarrow A^3\rarrow0$ be a short exact
sequence in~$\sA$.
 Then there exists a $\kappa$\+directed poset $\Xi$ and
a $\Xi$\+indexed diagram of short exact sequences $0\rarrow S^1_\xi
\rarrow S^2_\xi\rarrow S^3_\xi\rarrow0$ with $S^i_\xi\in\sA_{<\kappa}$
whose $\kappa$\+directed colimit over $\xi\in\Xi$ is the short exact
sequence $0\rarrow A^1\rarrow A^2 \rarrow A^3\rarrow0$.

 Now assume that $A^1\in\sA_{<\kappa}$ and $A^3\in\sA_{<\kappa}$.
 Then, for $i=1$ and~$3$, the isomorphism $A^i\rarrow
\varinjlim_{\xi\in\Xi}S^i_\xi$ factorizes through some objects in
the diagram.
 So there exists an index $\eta\in\Xi$ for which both the morphisms
$S^1_\eta\rarrow A^1$ and $S^3_\eta\rarrow A^3$ are split
epimorphisms in~$\sA_{<\kappa}$.
 As all split epimorphisms are admissible (since $\sA$ is weakly
idempotent-complete), and an extension of two admissible epimorphisms
in an exact category is an admissible
epimorphism~\cite[Corollary~3.6]{Bueh}, it follows that the morphism
$S^2_\eta\rarrow A^2$ is also an admissible epimorphism in~$\sA$.

 For every $i=1$, $2$, $3$, denote by $B^i$ the kernel of
the admissible epimorphism $S^i_\eta\rarrow A^i$.
 As the kernel of a termwise admissible epimorphism of short exact
sequences in an exact category is a short exact
sequence~\cite[Corollary~3.6]{Bueh}, we get a short exact sequence
$0\rarrow B^1\rarrow B^2\rarrow B^3\rarrow0$ in~$\sA$.
 Now a direct summand of a $\kappa$\+presentable object is
$\kappa$\+presentable, so $B^1$ and $B^3$ belong to~$\sA_{<\kappa}$.

 Applying the same argument to the short exact sequence $0\rarrow B^1
\rarrow B^2\rarrow B^3\rarrow0$, we obtain a termwise admissible
epimorphism from a short exact sequence $0\rarrow T^1\rarrow T^2
\rarrow T^3\rarrow0$ with $T^i\in\sA_{<\kappa}$ to the short
exact sequence $0\rarrow B^1\rarrow B^2\rarrow B^3\rarrow0$.
 Finally, we can conclude that the object $A^2$ belongs to
$\sA_{<\kappa}$, because $A^2$ is the cokernel of the composition
$T^2\rarrow B^2\rarrow S^2_\eta$ and finite colimits of
$\kappa$\+presentable objects are
$\kappa$\+presentable~\cite[Proposition~1.16]{AR}.
\end{proof}

\begin{lem} \label{admissible-epis-characterized}
 Let\/ $\sA$ be a locally $\kappa$\+coherent exact category and
$D\rarrow E$ be a morphism in\/~$\sA$.
 Then the following conditions are equivalent:
\begin{enumerate}
\item $D\rarrow E$ is an admissible epimorphism in\/~$\sA$;
\item for any object $S\in\sA_{<\kappa}$, any morphism $S\rarrow E$
in\/ $\sA$ can be included into a commutative square diagram
$$
 \xymatrix{
  D \ar[r] & E \\
  T \ar@{..>}[u] \ar@{..>>}[r] & S \ar[u]
 }
$$
\end{enumerate}
with an object $T\in\sA_{<\kappa}$ and an admissible epimorphism
$T\rarrow S$ in\/~$\sA$.
\end{lem}

\begin{proof}
 (1)~$\Longrightarrow$~(2)
 By the definition, there exists a $\kappa$\+directed poset $\Xi$ and
a $\Xi$\+indexed diagram of short exact sequences $0\rarrow U_\xi
\rarrow T_\xi\rarrow S_\xi\rarrow0$ in $\sA$ with the terms
$U_\xi$, $T_\xi$, $S_\xi$ belonging to $\sA_{<\kappa}$ such that
the morphism $D\rarrow E$ is the colimit of the morphisms
$T_\xi\rarrow S_\xi$ over $\xi\in\Xi$.
 Now the morphism $S\rarrow E$ factorizes as $S\rarrow S_\xi\rarrow E$
for some $\xi\in\Xi$.
 It remains to consider the pullback diagram
$$
 \xymatrix{
  U_\xi \ar@{>->}[r] & T_\xi \ar@{->>}[r] & S_\xi \\
  U_\xi \ar@{>->}[r] \ar@{=}[u] & T \ar@{->>}[r] \ar[u] & S \ar[u]
 }
$$
where the pullback object $T$ exists in $\sA$, since all pullbacks of
admissible epimorphisms exist by the definition of an exact category.
 By Proposition~\ref{finitely-presentables-extension-closed}, we have
$T\in\sA_{<\kappa}$, since $U_\xi\in\sA_{<\kappa}$ and
$S\in\sA_{<\kappa}$.

 (2)~$\Longrightarrow$~(1)
 The following observation plays the key role.
 Let $T'\rarrow S'$ be a morphism in $\sA_{<\kappa}$, and let
$(T'\to S')\rarrow (D\to E)$ be a morphism of morphisms in~$\sA$.
 By~(2), there exist an object $T''\in\sA_{<\kappa}$, an admissible
epimorphism $T''\twoheadrightarrow S'$ in $\sA$, and a morphism of
morphisms $(T''\twoheadrightarrow S')\rarrow (D\to E)$ in~$\sA$.
 Then the morphism $T'\oplus T''\rarrow S'$ with the components
$T'\rarrow S'$ and $T''\rarrow S'$ is an admissible epimorphism
in~$\sA$ by the dual version of~\cite[Exercise~3.11(i)]{Bueh}.
 There is a commutative triangular diagram of morphisms of morphisms
$(T'\to S')\rarrow (T'\oplus\nobreak T''\twoheadrightarrow
\nobreak S')\rarrow(D\to E)$; so the morphism $(T'\to S')\rarrow
(D\to E)$ factorizes through the object $(T'\oplus\nobreak T''
\twoheadrightarrow\nobreak S')$ in the category of morphisms in~$\sA$.

 Now, for the given fixed morphism $D\rarrow E$, denote by $\Xi$
the small category formed by all the commutative squares as in~(2)
with $T$, $S\in\sA_{<\kappa}$ and admissible epimorphisms
$T\rarrow S$.
 Then $\Xi$ is a $\kappa$\+filtered category and the colimit of
the morphisms $T_\xi\rarrow S_\xi$ over all $\xi\in\Xi$ is
the morphism $D\rarrow E$.
 These assertions follow from
Proposition~\ref{simple-diagrams-accessible}(a) and
(the proof of) Proposition~\ref{accessible-subcategory} applied
to the $\kappa$\+accessible category $\bM^2(\sA)$ of morphisms in $\sA$
and the class $\sT$ of all admissible epimorphisms in $\sA$ between
objects from~$\sA_{<\kappa}$.

 Finally, the class of all admissible epimorphisms in $\sA$ is
preserved by $\kappa$\+directed colimits in view of
Lemma~\ref{short-exact-preserved-by-directed-colimits}.
\end{proof}

\begin{prop} \label{finitely-presentables-adm-co-kernel-closed}
 In any locally $\kappa$\+coherent exact structure on
a $\kappa$\+accessible additive category $\sA$, the full subcategory
of $\kappa$\+presentable objects\/ $\sA_{<\kappa}\subset\sA$ is closed
under the kernels of admissible epimorphisms and the cokernels of
admissible monomorphisms in\/~$\sA$.
\end{prop}

\begin{proof}
 The full subcategory of $\kappa$\+presentable objects is closed under
the cokernels of all morphisms (more generally, under all
$\kappa$\+small colimits) that exist in $\sA$, for any category~$\sA$
with $\kappa$\+directed colimits~\cite[Proposition~1.16]{AR}.
 The assertion concerning the kernels of admissible epimorphisms is
nontrivial (and needs the epimorphism to be admissible).

 Let $0\rarrow C\rarrow D\rarrow E\rarrow0$ be an (admissible) short
exact sequence in $\sA$ with $\kappa$\+presentable objects $D$ and~$E$.
 Put $S=E$, and consider the identity morphism $S\rarrow E$.
 By Lemma~\ref{admissible-epis-characterized}\,(1)\,$\Rightarrow$\,(2)
there exists a commutative diagram of a morphism of short exact
sequences
$$
 \xymatrix{
  C \ar@{>->}[r] & D \ar@{->>}[r] & E \\
  U \ar@{>->}[r] \ar[u] & T \ar@{->>}[r] \ar[u] & S \ar@{=}[u] 
 }
$$
with $T\in\sA_{<\kappa}$.
 Moreover, following the proof of
Lemma~\ref{admissible-epis-characterized}\,(1)\,$\Rightarrow$\,(2),
the diagram can be constructed in such a way that the object $U$
is $\kappa$\+presentable.
 By~\cite[Proposition~2.12\,(iv)\,$\Rightarrow$\,(ii)]{Bueh}, we
have a short exact sequence $0\rarrow U\rarrow C\oplus T\rarrow D
\rarrow0$ in~$\sA$.
 Now $C\oplus T\in\sA_{<\kappa}$ by
Proposition~\ref{finitely-presentables-extension-closed}, hence
$C\in\sA_{<\kappa}$.
\end{proof}

\begin{rems}
 (1)~Proposition~\ref{simple-diagrams-accessible}(c) together with
Proposition~\ref{accessible-subcategory} (applied to
the $\kappa$\+accessible category $\bC^3(\sA)$ and
the class $\sT$ of all short exact sequences in~$\sA_{<\kappa}$) provide
the following description of the class of all short exact sequences
in~$\sA$.
  A three-term complex in a locally $\kappa$\+coherent exact category
$\sA$ is a short exact sequence if and only if any morphism into it
from a three-term complex in $\sA_{<\kappa}$ factorizes through a short
exact sequence in~$\sA_{<\kappa}$.

 (2)~Moreover, Proposition~\ref{simple-diagrams-accessible}(b) together
with Proposition~\ref{accessible-subcategory}
(applied to the $\kappa$\+accessible category $\bM^3(\sA)$ and
the class $\sT$ of all short exact sequences in~$\sA_{<\kappa}$, viewed
as objects of $\bM^3(\sA)$) provide the following description.
 A pair of composable morphisms in $\sA$ is a short exact sequence if
and only if any morphism into it from a pair of composable morphisms
in $\sA_{<\kappa}$ factorizes through a short exact sequence
in~$\sA_{<\kappa}$.

 (3)~Remark~\ref{accessible-monos-epis-as-dir-colims-remark} together
with Proposition~\ref{finitely-presentables-extension-closed}
(see also Proposition~\ref{finitely-presentables-adm-co-kernel-closed})
imply that the admissible epimorphisms in~$\sA$ are precisely
the $\kappa$\+directed colimits of admissible epimorphisms in
(or between the objects of) $\sA_{<\kappa}$, and the admissible
monomorphisms in $\sA$ are precisely the $\kappa$\+directed
colimits of admissible monomorphisms in (or between the objects
of)~$\sA_{<\kappa}$.
 Hence one can similarly characterize the admissible epimorphisms and
the admissible monomorphisms in $\sA$ by factorization properties.
 In particular, we obtain the following lemma.
\end{rems}

\begin{lem} \label{admissible-monos-characterized}
 Let\/ $\sA$ be a locally $\kappa$\+coherent exact category and
$C\rarrow D$ be a morphism in\/~$\sA$.
 Then the following conditions are equivalent:
\begin{enumerate}
\item $C\rarrow D$ is an admissible monomorphism in\/~$\sA$;
\item any morphism into $C\rarrow D$ from a morphism in\/
$\sA_{<\kappa}$ factorizes through an admissible monomorphism in
(the inherited exact category structure on)\/~$\sA_{<\kappa}$.
\end{enumerate}
\end{lem}

\begin{proof}
 (1)~$\Longrightarrow$~(2)
 By definition, $C\rarrow D$ is a $\kappa$\+directed colimit of
admissible monomorphisms in~$\sA_{<\kappa}$.
 It remains to use Proposition~\ref{simple-diagrams-accessible}(a)
to the effect that all morphisms between objects from $\sA_{<\kappa}$
are $\kappa$\+presentable in $\bM^2(\sA)$.

 (2)~$\Longrightarrow$~(1)
 Applying Proposition~\ref{accessible-subcategory} to
the $\kappa$\+accessible category of morphisms in $\sA$, we conclude
that $C\rarrow D$ is a $\kappa$\+directed colimit of admissible
monomorphisms in~$\sA_{<\kappa}$.
 It follows immediately by the definition of a locally
$\kappa$\+coherent exact category structure that $C\rarrow D$ is
an admissible monomorphism in~$\sA$.
\end{proof}

\begin{lem} \label{sharply-larger-lemma}
 Let\/ $\lambda$ and\/ $\mu$ be a pair of regular cardinals such that\/
$\lambda\triangleleft\mu$ in the sense of\/~\cite[Definition~2.12]{AR}.
 Then any\/ locally\/ $\lambda$\+coherent exact category is locally\/
$\mu$\+coherent.
\end{lem}

\begin{proof}
 Any $\lambda$\+accessible category is $\mu$\+accessible
by~\cite[Theorem~2.11]{AR}.
 The class of all admissible short exact sequences in a locally
$\lambda$\+coherent exact category\/ $\sA$ is closed under
$\lambda$\+directed colimits by 
Lemma~\ref{short-exact-preserved-by-directed-colimits}; hence it is
also closed under $\mu$\+directed colimits.
 It remains to obtain any short exact sequence $0\rarrow C\rarrow D
\rarrow E\rarrow0$ in $\sA$ as a $\mu$\+directed colimit of short
exact sequences of $\mu$\+presentable objects.
 The construction from~\cite[proof of
Theorem~2.11\,(iv)\,$\Rightarrow$\,(i)]{AR} fulfills the task,
producing the desired representation from a given representation
of $0\rarrow C\rarrow D\rarrow E\rarrow0$ as a $\lambda$\+directed
colimit of short exact sequences of $\lambda$\+presentable objects.
\end{proof}

\Section{The Induced Exact Category Structure on the Ind-Objects}
\label{ind-objects-secn}

 In this section, as in the previous one, $\kappa$~denotes
a regular cardinal.
 Let $\sS$ be a small additive category.
 Consider the category $\sA=\Ind_{(\kappa)}\sS$ of ind-objects
representable by $\kappa$\+directed diagrams of objects of~$\sS$.
 Then the additive category $\sA$ is $\kappa$\+accessible, and
the full subcategory of $\kappa$\+presentable objects in $\sA$ is
naturally equivalent to the idempotent completion $\overline\sS$ of
the category~$\sS$.
 Conversely, for any $\kappa$\+accessible additive category $\sA$,
and any full additive subcategory $\sS\subset\sA_{<\kappa}$ such that
all the objects of $\sA_{<\kappa}$ are direct summands of objects
from $\sS$, one has a natural equivalence of categories $\sA\simeq
\Ind_{(\kappa)}\sS$.
 In the case of the countable cardinal $\kappa=\aleph_0$, we will use
the notation $\Ind\sS$ instead of $\Ind_{(\aleph_0)}\sS$.

 The aim of this section is to show that any exact structure on $\sS$
extends uniquely to a locally $\kappa$\+coherent exact structure
on~$\sA$ with the property that the full subcategory $\sS$ is closed
under extensions in~$\sA$.
 By Proposition~\ref{finitely-presentables-extension-closed}, the latter
condition holds automatically when the category $\sS$ is
idempotent-complete.
 In particular, for any $\kappa$\+accessible additive category $\sA$,
locally $\kappa$\+coherent exact structures on $\sA$ correspond
bijectively to exact structures on the (essentially small) additive
category~$\sA_{<\kappa}$.

 We proceed in two steps: first, extend a given exact category structure
on $\sS$ to an exact category structure on the idempotent completion
$\overline\sS=\sA_{<\kappa}$ of $\sS$, and then extend it further to
a locally $\kappa$\+coherent exact structure on
$\sA=\Ind_{(\kappa)}\sS$.
 But first of all we present a counterexample demonstrating a pitfall
involved with the question of uniqueness of extension of exact
structures from $\sS$ to~$\overline\sS$.

\begin{ex}
 Let $\sB$ denote the abelian category of morphisms of
finite-dimensional vector spaces $f\:V_0\rarrow V_1$ over a field~$k$.
 Let $\sS\subset\sB$ be the full additive subcategory whose objects are 
all the morphisms~$f$ satisfying the equation $\dim_k V_0-\dim_k V_1=
\rk f$ (where $\rk f=\dim_k\im f$ is the rank of~$f$).
 Then $\sS$ is a weakly idempotent-complete additive category and all
objects of $\sB$ are direct summands of objects from~$\sS$.
 Still all short exact sequences in $\sB$ with all the three terms
belonging to~$\sS$ are split (as the equation $\rk g=\rk f+\rk h$
implies splitness of a short exact sequence $0\rarrow f\rarrow g
\rarrow h\rarrow0$ in~$\sB$).
 So both the abelian exact structure and the split exact structure on
$\sB$ restrict to the split exact structure on~$\sS$.
 Notice that $\sS$ is \emph{not} closed under extensions in
the abelian exact structure on $\sB$, however.
\end{ex}

\begin{lem} \label{idempotent-completion-of-exact-category}
 Let $\sE$ be an exact category and $\overline\sE$ be the idempotent
completion of the additive category~$\sE$.
 Then there is a unique exact category structure on $\overline\sE$
such that $\sE$ is closed under extensions in $\overline\sE$ and
the inherited exact category structure on $\sE$ is the originally
given exact category structure.
 Specifically, the exact category structure on $\overline\sE$ is defined
by the condition that the short exact sequences in $\overline\sE$ are
the direct summands of short exact sequences in~$\sE$.
\end{lem}

\begin{proof}
 Uniqueness: let $0\rarrow \overline A\rarrow \overline B\rarrow
\overline C\rarrow0$ be a short exact sequence in~$\overline\sE$.
 Then there exist objects $\overline A'$ and $\overline C'\in
\overline\sE$ such that $A=\overline A\oplus\overline A'\in\sE$
and $C=\overline C\oplus\overline C'\in\sE$.
 Now we have a short exact sequence $0\rarrow A\rarrow\overline A'\oplus
\overline B\oplus\overline C'\rarrow C\rarrow0$ in~$\overline\sE$.
 Since the full subcategory $\sE$ is closed under extensions in
$\overline\sE$ by assumption, it follows that all the three terms of
the new short exact sequence belong to~$\sE$.
 So any short exact sequence in $\overline\sE$ is a direct summand of
a short exact sequence in~$\sE$.
 Conversely, in any exact category, any direct summand of a short
exact sequence is a short exact sequence, as one can easily show using
the pullback and pushout axioms.

 Existence: define the short exact sequences in $\overline\sE$ to be
the direct summands of short exact sequences in~$\sE$.
 This endows $\overline\sE$ with an exact category structure satisfying
all the conditions~\cite[Proposition~6.13]{Bueh}.
 In particular, the fact that $\sE$ is closed under extensions in
$\overline\sE$ is provable using the pullback and pushout axioms
in the exact category~$\sE$.
\end{proof}

\begin{lem} \label{admissible-epis-in-the-colimit-exact-structure}
 Let\/ $\sA$ be a $\kappa$\+accessible additive category.
 Suppose given an exact category structure on\/~$\sA_{<\kappa}$.
 Let $D\rarrow E$ be a morphism in\/~$\sA$.
 Then the following conditions are equivalent:
\begin{enumerate}
\item $D\rarrow E$ is a $\kappa$\+directed colimit in\/ $\sA$ of
admissible epimorphisms in\/~$\sA_{<\kappa}$;
\item for any object $S\in\sA_{<\kappa}$, any morphism $S\rarrow E$
in\/ $\sA$ can be included into a commutative square diagram
$$
 \xymatrix{
  D \ar[r] & E \\
  T \ar@{..>}[u] \ar@{..>>}[r] & S \ar[u]
 }
$$
\end{enumerate}
with an object $T\in\sA_{<\kappa}$ and an admissible epimorphism
$T\rarrow S$ in\/~$\sA_{<\kappa}$.
\end{lem}

\begin{proof}
 The proof is similar to that of
Lemma~\ref{admissible-epis-characterized}.
 The only differences are that, in the proof of
(1)\,$\Longrightarrow$\,(2), the pullback should be taked in
the exact category~$\sA_{<\kappa}$;
and in the proof of (2)\,$\Longrightarrow$\,(1), the result
of~\cite[Exercise~3.11(i)]{Bueh} should be applied in
the exact category~$\sA_{<\kappa}$.
\end{proof}

\begin{lem} \label{pullback-pushforward-data-as-filtered-colimits}
 Let\/ $\sA$ be a $\kappa$\+accessible additive category.
 Suppose given an exact category structure on\/~$\sA_{<\kappa}$.
 Then \par
\textup{(a)} any diagram
$$
 \xymatrix{
  C \ar@{>->}[r] & D \ar@{->>}[r] & E \\
  && F \ar[u]
 }
$$
in the category\/ $\sA$ whose upper line is a $\kappa$\+directed colimit
in\/ $\sA$ of short exact sequences in\/ $\sA_{<\kappa}$ can be obtained
as a $\kappa$\+directed colimit in\/ $\sA$ of similar diagrams in\/
$\sA_{<\kappa}$ with short exact sequences in the upper line; \par
\textup{(b)} any diagram
$$
 \xymatrix{
  C \ar@{>->}[r] \ar[d] & D \ar@{->>}[r] & E \\
  B
 }
$$
in the category\/ $\sA$ whose upper line is a $\kappa$\+directed colimit
in\/ $\sA$ of short exact sequences in\/ $\sA_{<\kappa}$ can be obtained
as a $\kappa$\+directed colimit in\/ $\sA$ of similar diagrams in\/
$\sA_{<\kappa}$ with short exact sequences in the upper line.
\end{lem}

\begin{proof}
 Denote by $\sE$ the full subcategory in $\bC^3(\sA)$ consising of
all the $\kappa$\+directed colimits of short exact sequences
in~$\sA_{<\kappa}$.
 By Propositions~\ref{simple-diagrams-accessible}(c)
and~\ref{accessible-subcategory}, \,$\sE$~is a $\kappa$\+accessible
category, and its $\kappa$\+presentable objects are the short exact
sequences in $\sA_{<\kappa}$ (since the class of all short exact
sequences in $\sA_{<\kappa}$ is closed under direct summands
in~$\bC^3(\sA)$).
 Now both the parts~(a) and~(b) are provable by applying
Proposition~\ref{comma-accessible} to a suitable pair of functors.
 For part~(a), take $\sK=\sM=\sA$, \,$\sL=\sE$, the identity functor
$F\:\sK\rarrow\sM$, and the functor $G\:\sL\rarrow\sM$ assigning
the object $E$ to a three-term complex $C\rarrow D\rarrow E$.
 For part~(b), take $\sK=\sE$, \,$\sL=\sM=\sA$, the identity functor
$G\:\sL\rarrow\sM$, and the functor $F\:\sK\rarrow\sM$ assigning
the object $C$ to a three-term complex $C\rarrow D\rarrow E$.
\end{proof}

\begin{prop} \label{colimit-exact-structure-prop}
 Let\/ $\sA$ be a $\kappa$\+accessible additive category.
 Suppose given an exact category structure on\/~$\sA_{<\kappa}$.
 Then the class of all $\kappa$\+directed colimits in\/ $\sA$ of short
exact sequences in\/ $\sA_{<\kappa}$ is an exact category structure
on\/~$\sA$.
 Any short exact sequence in this exact category structure on\/ $\sA$
whose terms belong to\/ $\sA_{<\kappa}$ is a short exact sequence in
the original exact category structure on\/~$\sA_{<\kappa}$.
 Consequently, the resulting exact structure on\/ $\sA$ is locally
$\kappa$\+coherent.
\end{prop}

\begin{proof}
 To begin with, notice that the inclusion
$\sA_{<\kappa}\rarrow\sA$ preserves all cokernels that happen to exist
in the category~$\sA_{<\kappa}$.
 Indeed, let $g\:T\rarrow S$ be a cokernel of a morphism
$f\:U\rarrow T$ in $\sA_{<\kappa}$.
 Consider an arbitrary object $B\in\sA$.
 Then there exists a $\kappa$\+directed diagram
$(B_\upsilon)_{\upsilon\in\Upsilon}$ in $\sA_{<\kappa}$ such that
$B=\varinjlim B_\upsilon$.
 For any object $V\in\sA_{<\kappa}$, we have $\Hom_\sA(V,B)\simeq
\varinjlim_{\upsilon\in\Upsilon}\Hom_\sA(V,B_\upsilon)$.
 By assumption, the map $\Hom_\sA(g,B_\upsilon)$ is the kernel of
the map $\Hom_\sA(f,B_\upsilon)$ in the category of abelian groups
$\Ab$ for every $\upsilon\in\Upsilon$.
 Using the fact that directed colimits commute with kernels in $\Ab$,
one concludes that the map $\Hom_\sA(g,B)$ is the kernel of the map
$\Hom_\sA(f,B)$ in~$\Ab$.
 So the morphism~$g$ is a cokernel of the morphism~$f$ in~$\sA$.

 Now let $(0\to U_\xi\to T_\xi\to S_\xi\to0)_{\xi\in\Xi}$
be a $\kappa$\+directed diagram of short exact sequences in
$A_{<\kappa}$, indexed by a $\kappa$\+directed poset~$\Xi$.
 Let $0\rarrow C\rarrow D\rarrow E\rarrow0$ be the related short
sequence of $\kappa$\+directed colimits in~$\sA$.
 Denote the morphisms involved by $f_\xi\:U_\xi\rarrow T_\xi$ and
$g_\xi\:T_\xi\rarrow S_\xi$, and put $f=\varinjlim_{\xi\in\Xi}
f_\xi\:C\rarrow D$ and $g=\varinjlim_{\xi\in\Xi}g_\xi\:D\rarrow E$.
 Then the morphism~$g$ is the cokernel of the morphism~$f$ in
the category $\sA$, because the morphisms~$g_\xi$ are the cokernels
of the morphisms~$f_\xi$ in $\sA_{<\kappa}$ (hence also in~$\sA$),
and any existing colimit functors preserve any existing cokernels
in any additive category.

 To show that the morphism $f\:C\rarrow D$ is the kernel of
the morphism $g\:D\rarrow E$ in the category $\sA$, consider
an arbitrary object $B\in\sA$.
 Then there exists a ($\kappa$\+directed) diagram
$(B_\upsilon)_{\upsilon\in\Upsilon}$ in the category $\sA_{<\kappa}$
such that $B=\varinjlim_{\upsilon\in\Upsilon}B_\upsilon$.
 For any object $A\in\sA$, we have $\Hom_\sA(B,A)\simeq
\varprojlim_{\upsilon\in\Upsilon}\Hom_\sA(B_\upsilon,A)$.
 On the other hand, the functor $\Hom_\sA(B_\upsilon,{-})\:
\sA\rarrow\Ab$ preserves $\kappa$\+directed colimits for every
$\upsilon\in\Upsilon$.
 By assumption, the map $\Hom_\sA(B_\upsilon,f_\xi)$ is the kernel
of the map $\Hom_\sA(B_\upsilon,g_\xi)$ in $\Ab$ for every
$\upsilon\in\Upsilon$ and $\xi\in\Xi$.
 Since all limits and directed colimits commute with kernels in $\Ab$,
we can conclude that the map $\Hom_\sA(B,f)$ is the kernel of the map
$\Hom_\sA(B,g)$ in~$\Ab$.
 So the morphism~$f$ is the kernel of the morphism~$g$ in~$\sA$.

 It is clear from
Lemma~\ref{admissible-epis-in-the-colimit-exact-structure} that
the class of all admissible epimorphisms in the would-be exact
structure on $\sA$ is closed under compositions.
 The pullback and pushout axioms for short exact sequences in
$\sA$ are provable using
Lemma~\ref{pullback-pushforward-data-as-filtered-colimits}.

 Here, in the case of the pullbacks, one also needs to use the fact
that the $\kappa$\+directed colimits of objects from $\sA_{<\kappa}$
in $\sA$ take pullbacks in $\sA_{<\kappa}$ to pullbacks in~$\sA$
(more generally, this holds for all $\kappa$\+small limits that happen
to exist in~$\sA_{<\kappa}$).
 In the case of the pushouts, the fact that the inclusion
$\sA_{<\kappa}\rarrow\sA$ preserves those pushouts (and more generally,
$\kappa$\+small colimits) that happen to exist in $\sA_{<\kappa}$
needs to be used.
 Alternatively, it is convenient to check the axioms Ex2(a\+-b)
from~\cite[Section~A.3]{Partin} and use~\cite[Corollary~A.3]{Partin}.

 The second assertion of the proposition follows from
Proposition~\ref{simple-diagrams-accessible}(c) and
Proposition~\ref{accessible-subcategory} (applied to
the $\kappa$\+accessible category $\bC^3(\sA)$ and the class $\sT$
of all short exact sequences in~$\sA_{<\kappa}$).
 Here one needs to use the observation that the class of all short exact
sequences in $\sA_{<\kappa}$ (as in any exact category) is closed under
direct summands, as mentioned in the proofs of
Lemmas~\ref{idempotent-completion-of-exact-category}
and~\ref{pullback-pushforward-data-as-filtered-colimits}.
 The third and last assertion of the proposition follows from
the previous ones.
\end{proof}

\begin{cor} \label{ind-objects-loc-coh-exact-structure}
 For any small exact category\/ $\sS$, there exists a unique locally
$\kappa$\+coherent exact structure on the additive category of
ind-objects\/ $\Ind_{(\kappa)}\sS$ such that\/ $\sS$ is closed under
extensions in this exact category structure on\/ $\Ind_{(\kappa)}\sS$
and the inherited exact structure  on\/ $\sS$ is the given one.
\end{cor}

\begin{proof}
 Notice that $\sA=\Ind_{(\kappa)}\sS$ is a $\kappa$\+accessible
additive category with the full subcategory of $\kappa$\+presentable
objects $\sA_{<\kappa}$ naturally equivalent to the idempotent
completion~$\overline\sS$.
 Then use Lemma~\ref{idempotent-completion-of-exact-category},
Proposition~\ref{colimit-exact-structure-prop}, and
Proposition~\ref{finitely-presentables-extension-closed}.
\end{proof}

\begin{thm} \label{lc-exact-structures-bijective-correspondence}
 For any $\kappa$\+accessible additive category\/ $\sA$, there is
a natural bijective correspondence between exact structures on
the small additive category\/ $\sA_{<\kappa}$ and locally
$\kappa$\+coherent exact structures on the additive category\/~$\sA$.
 To any exact structure on\/ $\sA_{<\kappa}$, the exact structure on
$\sA$ given by the class of all $\kappa$\+directed colimits in\/ $\sA$
of short exact sequence in\/ $\sA_{<\kappa}$ is assigned.
 To any locally $\kappa$\+coherent exact structure on\/ $\sA$,
the inherited exact structure on the full subcategory\/
$\sA_{<\kappa}\subset\sA$ is assigned.
\end{thm}

\begin{proof}
 Follows from Propositions~\ref{colimit-exact-structure-prop}
and~\ref{finitely-presentables-extension-closed}.
\end{proof}

\Section{Example: Maximal Locally Coherent Exact Structure}
\label{example-maximal-subsecn}

 Notice that there is a natural partial order on the class of all exact
structures on any given additive category (given by inclusion of
classes of admissible short exact sequences).
 Moreover, the intersection of any nonempty set/class of classes of
short exact sequences defining exact structures is again a class
of short exact sequences defining an exact structure.
 Furthermore, any additive category has a unique maximal exact
category structure~\cite{SW,Cri,Rum}.
 Hence the set of all exact structures on a small additive category is
a complete lattice.

 Let $\kappa$~be a regular cardinal and $\sA$ be a $\kappa$\+accessible
additive category.
 Then Theorem~\ref{lc-exact-structures-bijective-correspondence}
provides an order-preserving bijection between the poset of locally
$\kappa$\+coherent exact structures on $\sA$ and the poset of all
exact structures on the essentially small additive category
$\sA_{<\kappa}$ of all $\kappa$\+presentable objects in~$\sA$.
 Thus the poset of all locally $\kappa$\+coherent exact structures on
$\sA$ is also a complete lattice.
 In this section, we are interested in the \emph{maximal} locally
$\kappa$\+coherent exact structure on~$\sA$.

 Recall that a category $\sA$ is said to be \emph{locally
$\kappa$\+presentable}~\cite[Definition~1.17 and Theorem~1.20]{AR}
if $\sA$ is $\kappa$\+accessible and all colimits exist in~$\sA$.
 Locally $\aleph_0$\+presentable categories are called \emph{locally
finitely presentable}~\cite[Definition~1.9 and Theorem~1.11]{AR}.

 The following lemma describes the maximal exact structure on
the category of $\kappa$\+presentable objects $\sA_{<\kappa}$ of
a locally $\kappa$\+presentable abelian category~$\sA$.
 Given an additive category $\sE$ and two morphisms $f$, $g$ in
$\sE$, we denote by $\ker_\sE(g)$ and $\coker_\sE(f)$ the kernel
and cokernel of $g$ and~$f$ computed in the category~$\sE$.

 The following lemma was essentially already explained in
the beginning paragraphs of the proof of
Proposition~\ref{colimit-exact-structure-prop}.

\begin{lem} \label{maximal-on-kappa-presentables}
 Let\/ $\sA$ be a locally $\kappa$\+presentable abelian category
and $S\overset f\rarrow T\overset g\rarrow U$ be a composable pair
of morphisms in the full subcategory\/ $\sA_{<\kappa}\subset\sA$.
 In this context: \par
\textup{(a)} if $g=\coker_{\sA_{<\kappa}}(f)$, then one also has
$g=\coker_\sA(f)$; \par
\textup{(b)} if $f=\ker_{\sA_{<\kappa}}(g)$, then one also has
$f=\ker_\sA(g)$. \par
 Consequently, the exact category structure on\/ $\sA_{<\kappa}$
inherited from the abelian exact structure on\/ $\sA$ is
the maximal exact structure on\/~$\sA_{<\kappa}$.
\end{lem}

\begin{proof}
 Notice first of all that the full subcategory $\sA_{<\kappa}$ is
closed under extensions in~$\sA$ by~\cite[Lemma~A.4]{Sto0}
(the running assumption in~\cite{Sto0} that the category is
Grothendieck is not needed for this lemma, as it suffices to know
that the $\kappa$\+directed colimit functors are exact in any locally
$\kappa$\+presentable abelian category~\cite[Proposition~1.59]{AR}).
 So $\sA_{<\kappa}$ inherits an exact category structure from
the abelian exact structure of $\sA$, and the last assertion of
the lemma makes sense.

 Part~(a): notice that the full subcategory $\sA_{<\kappa}$ is closed
under all $\kappa$\+small colimits, and in particular, under cokernels
in~$\sA$ \,\cite[Proposition~1.16]{AR}.
 Consequently, one has $\coker_{\sA_{<\kappa}}(f)=\coker_\sA(f)$
for any morphism~$f$ in $\sA_{<\kappa}$.

 Part~(b): denote by $S'=\ker_\sA(g)$ the kernel of the morphism~$g$
computed in the category~$\sA$.
 Then the morphism~$f$ factorizes as $S\overset h\rarrow S'\rarrow T$.
 Now $S'$ is the sum of the images of the morphisms into $S'$ from
$\kappa$\+presentable objects $V\in\sA_{<\kappa}$.
 Any such morphism $V\rarrow S'$ factorizes through~$h$, since
the composition $V\rarrow S'\rarrow T$ factorizes through $S\rarrow T$,
and $S'\rarrow T$ is a monomorphism in~$\sA$.
 Hence the morphism~$h$ is an epimorphism in~$\sA$.

 Finally, the kernel $K$ of the morphism $S\rarrow S'$ in $\sA$
is the sum of the images of the morphisms into $K$ from
$\kappa$\+presentable objects $W\in\sA_{<\kappa}$.
 Any such morphism $W\rarrow K$ has to vanish, since the composition
$W\rarrow K\rarrow S\rarrow S'\rarrow T$ vanishes, while
the morphism $S\rarrow T$ is a monomorphism in $\sA_{<\kappa}$
and the morphism $K\rarrow S$ is a monomorphism in~$\sA$.
 Thus $K=0$ and $h$~is an isomorphism.
\end{proof}

 The following two corollaries describe the admissible epimorphisms
and the admissible monomorphisms in the maximal locally
$\kappa$\+coherent exact structure on~$\sA$.

\begin{cor} \label{maximal-admissible-epimorphisms-cor}
 Let\/ $\sA$ be a locally $\kappa$\+presentable abelian category and
$D\rarrow E$ be a morphism in\/~$\sA$.
 Then the following conditions are equivalent:
\begin{enumerate}
\item $D\rarrow E$ is an admissible epimorphism in the maximal
locally $\kappa$\+coherent exact structure on\/~$\sA$;
\item for any object $S\in\sA_{<\kappa}$, any morphism $S\rarrow E$
in\/ $\sA$ can be included into a commutative square diagram
$$
 \xymatrix{
  D \ar[r] & E \\
  T \ar@{..>}[u] \ar@{..>>}[r] & S \ar[u]
 }
$$
\end{enumerate}
such that $T\in\sA_{<\kappa}$ and $T\rarrow S$ is an epimorphism
in\/ $\sA$ with the kernel\/ $\ker(T\twoheadrightarrow S)$
belonging to\/~$\sA_{<\kappa}$.
\end{cor}

\begin{proof}
 Follows from
Lemmas~\ref{admissible-epis-in-the-colimit-exact-structure}
and~\ref{maximal-on-kappa-presentables}.
\end{proof}

\begin{cor} \label{maximal-admissible-monomorphisms-cor}
 Let $\sA$ be a locally $\kappa$\+presentable abelian category and
$C\rarrow D$ be a morphism in\/~$\sA$.
 Then the following conditions are equivalent:
\begin{enumerate}
\item $C\rarrow D$ is an admissible monomorphism in the maximal
locally $\kappa$\+coherent exact structure on\/~$\sA$;
\item any morphism into $C\rarrow D$ from a morphism in\/
$\sA_{<\kappa}$ factorizes through a morphism in\/ $\sA_{<\kappa}$
that is a monomorphism in\/~$\sA$.
\end{enumerate}
\end{cor}

\begin{proof}
 It is worth noticing that a morphism~$g$ in $\sA_{<\kappa}$ is
a monomorphism in $\sA$ if and only if $g$~is a monomorphism
in $\sA_{<\kappa}$ (by Lemma~\ref{maximal-on-kappa-presentables}(b)).
 Furthermore, any monomorphism is admissible in the maximal exact
structure on $\sA_{<\kappa}$, as the full subcategory $\sA_{<\kappa}$
is closed under cokernels in~$\sA$ (by~\cite[Proposition~1.16]{AR}).
 In view of these remarks, the assertion of the corollary follows
from Lemma~\ref{admissible-monos-characterized}.
\end{proof}

 A locally $\kappa$\+presentable abelian category $\sA$ is said to be
\emph{locally\/ $\kappa$\+coherent} (as an abelian category) if
the kernel of any epimorphism  between two $\kappa$\+presentable
objects is again $\kappa$\+presentable in~$\sA$.
 Equivalently, $\sA$ is locally $\kappa$\+coherent if and only if
the full subcategory $\sA_{<\kappa}$ is closed under kernels (of
arbitrary morphisms) in~$\sA$.
 In the case of the cardinal $\kappa=\aleph_0$, this definition agrees
with the definition of a locally coherent abelian category introduced
in the paper~\cite[Section~2]{Roo} and discussed
in~\cite[Section~9.5]{Pedg}, \cite[Section~8.2]{PS5}.

\begin{cor} \label{locally-coherent-abelian-corollary}
 Let\/ $\sA$ be a locally $\kappa$\+presentable abelian category.
 Then the maximal locally $\kappa$\+coherent exact structure on\/ $\sA$
coincides with the abelian exact structure on\/ $\sA$ if and only if\/
$\sA$ is locally $\kappa$\+coherent as an abelian category.
\end{cor}

\begin{proof}
 ``Only if'': by
Proposition~\ref{finitely-presentables-adm-co-kernel-closed},
the full subcategory $\sA_{<\kappa}$ is closed under kernels of
admissible epimorphisms in any locally $\kappa$\+coherent exact
structure on~$\sA$.

 ``If'': when $\sA$ is a locally $\kappa$\+coherent abelian category,
the full subcategory $\sA_{<\kappa}$ is closed under kernels and
cokernels in~$\sA$; so $\sA_{<\kappa}$ is an abelian category.
 We have to show that the locally $\kappa$\+coherent exact structure
on $\sA$ corresponding to the abelian exact structure on
$\sA_{<\kappa}$ is the abelian exact structure on~$\sA$.

 For this purpose, in view of
Corollary~\ref{maximal-admissible-epimorphisms-cor}, it suffices to
check that any epimorphism $D\rarrow E$ in $\sA$ and any morphism
$S\rarrow E$ in $\sA$ with $S\in\sA_{<\kappa}$ can be included into
a commutative square diagram
$$
 \xymatrix{
  D \ar[r] & E \\
  T \ar@{..>}[u] \ar@{..>>}[r] & S \ar[u]
 }
$$
with an object $T\in\sA_{<\kappa}$ and an epimorphism $T\rarrow S$
in~$\sA$.
 The latter assertion holds in any locally $\kappa$\+presentable
abelian category~$\sA$, as explained in~\cite[proof of Lemma~10.7]{Pacc}.
\end{proof}

\begin{ex}
 To give another example of a locally $\kappa$\+coherent exact
category structure, consider a locally $\kappa$\+presentable
abelian category $\sA$, and suppose given a class of
$\kappa$\+presentable objects $\sT\subset\sA_{<\kappa}$ such that
$\sT$ is closed under extensions in $\sA_{<\kappa}$, or
equivalently, in~$\sA$ (cf.\ the proof of
Lemma~\ref{maximal-on-kappa-presentables}).

 Then, by Proposition~\ref{accessible-subcategory}, the full
subcategory $\sB=\varinjlim_{(\kappa)}\sT\subset\sA$ is
$\kappa$\+accessible, and the full subcategory of all
$\kappa$\+presentable objects in $\sB$ coincides with the idempotent
completion of $\sT$, that is $\sB_{<\kappa}=\overline\sT$.
 Following the discussion in the beginnning of
Section~\ref{ind-objects-secn}, we have an equivalence of
additive categories $\sB\simeq\Ind_{(\kappa)}\sT$.

 Endow the additive category $\sT$ with the exact category structure
inherited from the abelian category~$\sA$.
 Then the additive category $\sB\simeq\Ind_{(\kappa)}\sT$ acquires
the induced locally $\kappa$\+coherent exact structure, as per
Corollary~\ref{ind-objects-loc-coh-exact-structure}.
 We will call this exact structure the \emph{standard locally
$\kappa$\+coherent exact structure} on $\sB=\varinjlim_{(\kappa)}\sT$.

 Notice that, even when $\sT=\sA_{<\kappa}$, and consequently,
$\sB=\sA$, the standard locally $\kappa$\+coherent exact structure
on $\sA$ differs from the abelian exact structure (generally speaking).
 In fact, the standard locally $\kappa$\+coherent exact structure
on $\sB=\sA$ is the maximal locally $\kappa$\+coherent exact structure
(by Lemma~\ref{maximal-on-kappa-presentables}).
 For example, when $\kappa=\aleph_0$ and $\sA=\Modr R$ is
the category of right modules over a ring $R$, the maximal locally
coherent exact structure on $\Modr R$ coincides with the abelian
exact structure if and only if the ring $R$ is right coherent
(by Corollary~\ref{locally-coherent-abelian-corollary}).
\end{ex}

\Section{Example: Pure Exact Structure and Flat Coherence}
\label{pure-flat-subsecn}

 Let $\kappa$~be a regular cardinal and $\sA$ be a $\kappa$\+accessible
additive category.
 The minimal element in the complete lattice of all locally
$\kappa$\+coherent exact structures on $\sA$ (as per the discussion
in Section~\ref{example-maximal-subsecn}) is the locally
$\kappa$\+coherent exact structure corresponding to the split exact
structure on~$\sA_{<\kappa}$.
 This exact structure on a $\kappa$\+accessible additive category $\sA$
is called the \emph{$\kappa$\+pure exact structure}, or in the case
of $\kappa=\aleph_0$, simply the \emph{pure exact
structure}~\cite[Section~3]{CB}, \cite[Section~4]{Sto}.

 The admissible short exact sequences in the $\kappa$\+pure exact
structure on $\sA$ are the $\kappa$\+directed colimits of split
short exact sequences in $\sA_{<\kappa}$, or equivalently
(by Lemma~\ref{short-exact-preserved-by-directed-colimits}),
the $\kappa$\+directed colimits of split short exact sequences
in~$\sA$.
 So the $\kappa$\+pure exact structure is locally $\kappa$\+coherent
by definition, and in particular, the pure exact structure is
locally coherent.

 The following lemmas and propositions provide a description of
admissible epimorphisms and admissible monomorphisms in
the $\kappa$\+pure exact structure.
 These descriptions are fairly standard and usually taken as
the definitions of pure epimorphisms and pure monomorphisms
(cf.~\cite[Definition~2.27]{AR}, \cite[Definition~1]{AR2}, \cite{CB},
\cite{Sto}); nevertheless, our results provide a generalization
of~\cite[Proposition~5]{AR2} from abelian to additive categories.

\begin{lem} \label{pure-epimorphisms-lemma}
 Let\/ $\sA$ be a category and $p\:D\rarrow E$, \ $e\:S\rarrow E$ be
a pair of morphisms in\/ $\sA$ with the same codomain.
 Then there exists a commutative square
$$
 \xymatrix{
  D \ar[r]^p & E \\
  T \ar@{..>}[u]^d \ar@{..>>}[r]_s & S \ar[u]_e
 }
$$
with a split epimorphism $s\:T\rarrow S$ in $\sA$ if and only if
the morphism $e\:S\rarrow E$ factorizes through the morphism
$p\:D\rarrow E$,
$$
 \xymatrix{
  D \ar[r]^p & E \\
  & S \ar[u]_e \ar@{..>}[lu]^l
 }
$$
 If this is the case, one can choose $T=S$ and $s=\id_S$.
\end{lem}

\begin{proof}
 Given a splitting $i\:S\rarrow T$ (such that $si=\id_S$), it
suffices to put $l=di$.
 The remaining assertions are even more obvious.
\end{proof}

\begin{prop} \label{pure-epimorphisms-characterized-prop}
 Let\/ $\sA$ be a $\kappa$\+accessible additive category and
$p\:D\rarrow E$ be a morphism in\/~$\sA$.
 Then $p$~is an admissible epimorphism in the $\kappa$\+pure exact
structure on\/ $\sA$ if and only if, for every object
$S\in\sA_{<\kappa}$, any morphism $e\:S\rarrow E$ can be lifted
over the morphism~$p$,
$$
 \xymatrix{
  D \ar[r]^p & E \\
  & S \ar[u]_e \ar@{..>}[lu]^l
 }
$$
\end{prop}

\begin{proof}
 Compare Lemma~\ref{admissible-epis-in-the-colimit-exact-structure}
with Lemma~\ref{pure-epimorphisms-lemma}.
\end{proof}

\begin{lem} \label{pure-monomorphisms-lemma}
 Let\/ $\sA$ be a $\kappa$\+accessible additive category.
 Suppose given a commutative square diagram
$$
 \xymatrix{
  C \ar[r]^i & D \\
  S \ar[u]^c \ar[r]_t & T \ar[u]_d
 }
$$
in\/ $\sA$ with objects $S$, $T\in\sA_{<\kappa}$.
 Then the following two conditions are equivalent:
\begin{enumerate}
\item the morphism of morphisms $(c,d)\:(S\overset t\to T)\rarrow
(C\overset i\to D)$ factorizes through a split monomorphism
$(U\overset j\to V)$ in\/ $\sA$ with $U$, $V\in\sA_{<\kappa}$;
\item the morphism $c\:S\rarrow C$ factorizes through the morphism
$t\:S\rarrow T$,
$$
 \xymatrix{
  C \\
  S \ar[u]^c \ar[r]_t & T \ar@{..>}[lu]_e
 }
$$
\end{enumerate}
\end{lem}

\begin{proof}
 (1)\,$\Longrightarrow$\,(2) Suppose given a factorization
$$
 \xymatrix{
  C \ar[r]^i & D \\
  U \ar[u]^a \ar[r]^j & V \ar[u]_b \\
  S \ar[u]^u \ar[r]_t & T \ar[u]_v
 }
$$
with $au=c$ and $bv=d$ as in~(1), and let $q\:V\rarrow U$ be
a splitting (such that $qj=\id_U$).
 Then one has $aqvt=aqju=au=c$, so it remains to put $e=aqv\:
T\rarrow C$ in order to satisfy~(2).

 (2)\,$\Longrightarrow$\,(1) Consider the morphism $f=d-ie\:T\rarrow D$.
 One has $ft=dt-iet=ic-ic=0$, so the composition
$S\overset t\rarrow T\overset f\rarrow D$ vanishes.

 Let $D=\varinjlim_{\xi\in\Xi}V_\xi$ be a representation of $D$ as
a $\kappa$\+directed colimit of $\kappa$\+presentable objects~$V_\xi$.
 Let $b_\xi\:V_\xi\rarrow D$ be the morphisms in the colimit cocone.
 Since the object $T$ is $\kappa$\+presentable, there exists an index
$\xi_0\in\Xi$ such that the morphism $f\:T\rarrow D$ factorizes as
$T\rarrow V_{\xi_0}\overset{b_{\xi_0}}\rarrow D$.
 Then the composition $S\overset t\rarrow T\rarrow V_{\xi_0}
\overset{b_{\xi_0}}\rarrow D$ vanishes.
 Since the object $S$ is $\kappa$\+presentable, there exists an index
$\xi_1\in\Xi$, \,$\xi_1>\xi_0$, such that the composition
$S\overset t\rarrow T\rarrow V_{\xi_0}\rarrow V_{\xi_1}$ vanishes.
 Denote the composition $T\rarrow V_{\xi_0}\rarrow V_{\xi_1}$ by
$v'\:T\rarrow V_{\xi_1}$; so we have $v't=0$ and $b_{\xi_1}v'=f$.
 Now the commutative diagram
$$
 \xymatrix{
  C \ar[rr]^i && D \\
  T \ar[u]^e \ar[rr]^-{(\id_T,0)}
  && T\oplus V_{\xi_1} \ar[u]_{(ie,b_{\xi_1})} \\
  S \ar[u]^t \ar[rr]_t && T \ar[u]_{(\id_T,v')}
 }
$$
provides the desired factorization with $T\oplus V_{\xi_1}\in
\sA_{<\kappa}$ and a split monomorphism $j=(\id_T,0)$.
\end{proof}

\begin{prop} \label{pure-monomorphisms-characterized-prop}
 Let\/ $\sA$ be a $\kappa$\+accessible additive category and
$i\:C\rarrow D$ be a morphism in\/~$\sA$.
 Then $i$~is an admissible monomorphism in the $\kappa$\+pure exact
structure on\/ $\sA$ if and only if, for every commutative square
diagram
$$
 \xymatrix{
  C \ar[r]^i & D \\
  S \ar[u]^c \ar[r]_t & T \ar[u]_d
 }
$$
in\/ $\sA$ with objects $S$, $T\in\sA_{<\kappa}$, the morphism
$c\:S\rarrow C$ factorizes through the morphism $t\:S\rarrow T$,
$$
 \xymatrix{
  C \\
  S \ar[u]^c \ar[r]_t & T \ar@{..>}[lu]_e
 }
$$
\end{prop}

\begin{proof}
 Compare Lemma~\ref{admissible-monos-characterized}
with Lemma~\ref{pure-monomorphisms-lemma}.
\end{proof}

\begin{ex} \label{flat-modules-example}
 Let $R$ be a ring and $\sA=\Modrfl R$ be the category of flat
right $R$\+modules.
 Then $\sA$ is a finitely accessible additive category, and its
finitely presentable objects are precisely all the finitely
generated projective $R$\+modules (since the flat $R$\+modules
are the directed colimits of finitely generated projective ones,
by the classical Govorov--Lazard theorem~\cite{Gov,Laz}).

 Furthermore, the exact category structure on $\sA$ inherited from
the abelian exact structure of the ambient abelian category of
arbitrary $R$\+modules $\Modr R\supset\Modrfl R$ coincides with
the pure exact structure on~$\sA$.
 Indeed, the functors of directed colimit are exact in $\Modr R$,
so the directed colimits of split short exact sequences of
finitely generated projective modules are short exact sequences
of flat modules.
 Conversely, any surjective morphism of flat $R$\+modules is
an epimorphism in the pure exact structure on $\sA=\Modrfl R$, as one
can see from Proposition~\ref{pure-epimorphisms-characterized-prop}.

 The assertion that the pure exact structure on $\Modrfl R$ is
inherited from the abelian exact structure on $\Modr R$ is also
a special case of Theorem~\ref{canonical-embedding-theorem} below.
 See the first paragraph of Example~\ref{second-periodicity-example}.
\end{ex}

\begin{prop} \label{flat-modules-loc-kappa-coherent}
 For any ring $R$ and regular cardinal~$\kappa$, the exact category
of flat $R$\+modules\/ $\sA=\Modrfl R$ (with the exact structure
inherited from the abelian exact structure on $\Modr R$) is locally
$\kappa$\+coherent.
 The $\kappa$\+presentable objects of\/ $\Modrfl R$ are the flat
$R$\+modules that are $\kappa$\+presentable in\/ $\Modr R$.
\end{prop}

\begin{proof}
 For any uncountable regular cardinal~$\kappa$ one has $\aleph_0
\triangleleft\kappa$ \,\cite[Example~2.13(1)]{AR}.
 So the first assertion of the proposition follows from
Example~\ref{flat-modules-example} and
Lemma~\ref{sharply-larger-lemma}.
 The second assertion can be easily obtained
from~\cite[proof of Theorem~2.11(iv)\,$\Rightarrow$\,(i)]{AR};
see~\cite[Lemma~1.2]{Pflcc}.
\end{proof}

 The following corollary, which can be also found
in~\cite[Lemma~4.1]{Pflcc} for $\kappa=\aleph_1$ and
in~\cite[Corollary~10.12]{Pacc} in the general case, describes
the ``flat coherence'' phenomenon.

\begin{cor}
 Let $R$ be a ring and $\kappa$~be a regular cardinal.
 Then the kernel of any surjective morphism of $\kappa$\+presentable
flat $R$\+modules is a $\kappa$\+presentable flat $R$\+module.
\end{cor}

\begin{proof}
 Compare Proposition~\ref{flat-modules-loc-kappa-coherent}
with Proposition~\ref{finitely-presentables-adm-co-kernel-closed}.
\end{proof}

\Section{Locally Coherent Exact Categories are of Grothendieck Type}
\label{grothendieck-type-secn}

 Starting from this section, we gradually specialize our discussion
in this paper from the case of an arbitrary regular cardinal~$\kappa$
to that of $\kappa=\aleph_0$.
 In this context, we will call the locally $\aleph_0$\+coherent
exact categories in the sense of the definition in
Section~\ref{basic-properties-secn} simply \emph{locally coherent
exact categories}.

 The aim of this section is to show that all locally coherent exact
categories are \emph{of Grothendieck type} in the sense of
\v St\!'ov\'\i\v cek~\cite[Definitions~3.4 and~3.11]{Sto-ICRA}.
 Consequently, the small object argument can be used to prove
completeness of cotorsion pairs in such exact
categories~\cite[Theorem~5.16]{Sto-ICRA};
and in particular, any such exact category has enough injective
objects~\cite[Corollary~5.9]{Sto-ICRA}.

 Let $\sA$ be a category and $\alpha$~be an ordinal.
 An \emph{$\alpha$\+indexed chain} (of objects and morphisms)
$(A_\beta)_{0\le\beta<\alpha}$ in $\sA$ is a diagram in $\sA$ indexed
by the ordered set~$\alpha$.
 So $A_\beta$ is an object of $\sA$ for every $0\le\beta<\alpha$,
and a commutative diagram of morphisms $A_\beta\rarrow A_\gamma$
in $\sA$ is given, indexed by all pairs of ordinals $0\le\beta<
\gamma<\alpha$.
 An $\alpha$\+indexed chain $(A_\beta)_{0\le\beta<\alpha}$ in $\sA$ is
said to be \emph{smooth} if $A_\gamma=\varinjlim_{\beta<\gamma}A_\beta$
in $\sA$ for every limit ordinal $\gamma<\alpha$.
 If the colimit $A_\alpha=\varinjlim_{\beta<\alpha}A_\beta$ of a smooth
chain $(A_\beta)_{0\le\beta<\alpha}$ of objects and morphisms in $\sA$
exists, then the morphism $A_0\rarrow A_\alpha$ is said to be
the \emph{transfinite composition} of the morphisms $A_\beta\rarrow
A_{\beta+1}$, \ $0\le\beta<\beta+1<\alpha$.

 Let $\sM$ be a class of morphisms in~$\sA$.
 We will say that \emph{transfinite compositions of morphisms from\/
$\sM$ exist} in $\sA$ if, for every smooth chain
$(A_\beta)_{0\le\beta<\alpha}$ of objects and morphisms in $\sA$
such that the morphism $A_\beta\rarrow A_{\beta+1}$ belongs to $\sM$
for all $0\le\beta<\beta+1<\alpha$, the colimit $A_\alpha=
\varinjlim_{\beta<\alpha}A_\beta$ exists in~$\sA$.
 If this is the case, and moreover, the transfinite composition
$A_0\rarrow A_\alpha$ then also belongs to $\sM$, we will say that
transfinite compositions of morphisms from $\sM$ exists in $\sA$
and \emph{are themselves morphisms from\/~$\sM$}.

 Let $\sA$ be a category and $\sM$ be a class of morphisms in~$\sA$.
 Then an object $B\in\sA$ is said to be \emph{small relative to\/~$\sM$}
if there exists a cardinal~$\lambda$ satisfying the following
condition.
 For every smooth chain of objects and morphisms
$(A_\beta)_{0\le\beta<\alpha}$ in $\sA$ such that the morphisms
$A_\beta\rarrow A_{\beta+1}$ belong to $\sM$ for all
$0\le\beta<\beta+1<\alpha$ and the cofinality of the ordinal~$\alpha$
is not smaller than~$\lambda$, the natural map of
sets $\varinjlim_{\beta<\alpha}\Hom_\sA(B,A_\beta)\rarrow
\Hom_\sA(B,A_\alpha)$ must be bijective.
 It is presumed here that the colimit $A_\alpha=
\varinjlim_{\beta<\alpha}A_\beta$ exists in~$\sA$.

 Let $\sA$ be an exact category.
 An \emph{$\alpha$\+indexed filtration} $(F_\beta)_{0\le\beta\le\alpha}$
in $\sA$ is a smooth $(\alpha\nobreak+1)$\+indexed chain of objects
and morphisms such that $F_0=0$ and the morphism $F_\beta\rarrow
F_{\beta+1}$ is an admissible monomorphism for all $0\le\beta<\alpha$.
 Given an $\alpha$\+indexed filtration $(F_\beta)_{0\le\beta\le\alpha}$,
the object $F=F_\alpha$ is said to be \emph{filtered by} the cokernels
$U_\beta=\coker(F_\beta\to F_{\beta+1})$ of the morphisms
$F_\beta\rarrow F_{\beta+1}$, \ $0\le\beta<\alpha$.

 Given a class of objects $\sU\subset\sA$, one says that an object
$F\in\sA$ is \emph{filtered by} (objects from) $\sU$ if there is
an ordinal~$\alpha$ and an $\alpha$\+indexed filtration
$(F_\beta)_{0\le\beta\le\alpha}$ on $F=F_\alpha$ such that
$\coker(F_\beta\to F_{\beta+1})\in\sU$ for all $0\le\beta<\alpha$.
 The class of all objects in $\sA$ filtered by objects from $\sU$
is denoted by $\Fil(\sU)\subset\sA$.

 A class of objects $\sF\subset\sA$ is said to be \emph{deconstructible}
if there exists a \emph{set} of objects $\sU\subset\sF$ such that
$\sF=\Fil(\sU)$.
 In particular, the exact category $\sA$ is said to be
\emph{deconstructible in itself} if there exists a set of objects
$\sU\subset\sA$ such that $\sA=\Fil(\sU)$.

 Given an exact category $\sA$, we say that $\sA$ \emph{admits
a generator} if there is an object $G\in\sA$ such that for
every object $B\in\sA$ there exists a set $\Lambda$ and an admissible
epimorphism $G^{(\Lambda)}\rarrow B$.
 Here $G^{(\Lambda)}$ denotes the coproduct of $\Lambda$ copies of
$G$ in~$\sA$.

 An exact category $\sA$ is called
\emph{efficient}~\cite[Definition~3.4]{Sto-ICRA} if it satisfies
the following conditions:
\begin{description}
\item[(Ef0) or~(GT0)] The additive category $\sA$ is weakly
idempotent-complete.
\item[(Ef1) or~(GT1)] All transfinite compositions of admissible
monomorphisms exist in $\sA$, and are themselves admissible 
monomorphisms.
\item[(Ef2) or~(GT2)] Every object of $\sA$ is small relative to
the class of all admissible monomorphisms.
\item[(Ef3) or~(GT3)] The exact category $\sA$ admits a generator.
\end{description}

 An exact category $\sA$ is said to be \emph{of Grothendieck
type}~\cite[Definition~3.11]{Sto-ICRA} if $\sA$ is efficient
(i.~e., it satisfies conditions (GT0)--(GT3)), and the following
additional condition holds:
\begin{description}
\item[(GT4)] The exact category $\sA$ is deconstructible in itself.
\end{description}

 Given a finitely accessible category $\sA$, we will use the notation
$\sA_\fp=\sA_{<\aleph_0}$ for the full subcategory of finitely
presentable objects in~$\sA$.
 The following well-known lemma is a generalization
of~\cite[Lemma~2.2]{BHP} and~\cite[Proposition~5.2]{Pgen}.

\begin{lem} \label{pure-exact-structure-is-flat-modules}
 Let\/ $\sA$ be a finitely accessible additive category and $\cS$ be
a small category equivalent to\/ $\sA_\fp$, viewed as a ring with
many objects.
 Then the additive category\/ $\sA$ is naturally equivalent to
the additive category\/ $\Modrfl\cS$ of flat right $\cS$\+modules.
 The pure exact structure on\/ $\sA$ corresponds to the exact structure
on\/ $\Modrfl\cS$ inherited from the abelian category\/ $\Modr\cS$ of
all right $\cS$\+modules.
\end{lem}

\begin{proof}
 This is~\cite[Section~1.4]{CB} or~\cite[Proposition~4.2]{Sto},
and a particular case of~\cite[Theorem~2.26]{AR}.
 The point is that the category $\sA_\fp$ is naturally equivalent to
the category of finitely generated projective right $\cS$\+modules
(since $\sA_\fp$ is additive and idempotent-complete), and the flat
right $\cS$\+modules are the directed colimits of the finitely
generated projective ones.
 For the second assertion, cf.\ Example~\ref{flat-modules-example}.
\end{proof}

 Let $\sA$ be an exact category.
 We will say that $\sA$ has \emph{exact functors of directed colimit}
if all directed colimits exist in the additive category $\sA$ and
the directed colimits of admissible short exact sequences are
admissible short exact sequences.
 A locally $\kappa$\+coherent exact category \emph{need not} have
exact functors of directed colimit, but it always has exact functors
of $\kappa$\+directed colimit (in the obvious sense), by
Lemma~\ref{short-exact-preserved-by-directed-colimits}.

\begin{prop} \label{efficient-prop}
 Let $\sA$ be a $\kappa$\+accessible additive category with directed
colimits, endowed with an exact structure with exact functors of
directed colimit.
 Then $\sA$ is an efficient exact category.
\end{prop}

\begin{proof}
 We check the axioms one by one.

 (Ef0) or~(GT0) Any additive category with $\kappa$\+directed colimits
is idempotent-complete.

 (Ef1) or~(GT1) All transfinite compositions of chains of morphisms
exist in any category with directed colimits.
 All transfinite compositions of admissible monomorphisms are
admissible monomorphisms in any exact category with exact functors
of directed colimit (since finite compositions of admissible
monomorphisms are admissible monomorphisms in any exact category,
and directed colimits of admissible monomorphisms are admissible
monomorphisms in an exact category with exact functors of directed
colimit).

 (Ef2) or~(GT2) All objects in any accessible category have
presentability ranks, so they are small with respect to the class
of all morphisms.
 In fact, if $B=\varinjlim_{\xi\in\Xi}S_\xi$ in a category $\sA$ with
$\kappa$\+directed colimits, where $\Xi$ is a small $\kappa$\+filtered
category and $S_\xi\in\sA_{<\kappa}$ for all $\xi\in\Xi$, and if
$\lambda\ge\kappa$ is a regular cardinal greater than the cardinality
of (the set of all objects and morphisms in) $\Xi$, then the object $B$
is $\lambda$\+presentable in~$\sA$ (see~\cite[Proposition~1.16]{AR}).

 (Ef3) or~(GT3) The point is that any admissible epimorphism in
the $\kappa$\+pure exact structure on $\sA$ is also an admissible
epimorphism in any exact structure on $\sA$ with exact functors of
$\kappa$\+directed colimit.
 So it suffices to consider the case of the $\kappa$\+pure exact
structure on~$\sA$.

 Now the coproduct of all representatives of isomorphism
classes of $\kappa$\+presentable objects is a generator in
the $\kappa$\+pure exact structure on any $\kappa$\+accessible
additive category~$\sA$.
 In fact, if $B=\varinjlim_{\xi\in\Xi}S_\xi$ in $\sA$, where $\Xi$ is
a small $\kappa$\+filtered category and $S_\xi\in\sA_{<\kappa}$ for all
$\xi\in\Xi$, then the natural morphism $\coprod_{\xi\in\Xi}S_\xi\rarrow
\varinjlim_{\xi\in\Xi}S_\xi$ is a $\kappa$\+pure epimorphism in (i.~e.,
an admissible epimorphism in the $\kappa$\+pure exact structure
on)~$\sA$, as one can see from
Proposition~\ref{pure-epimorphisms-characterized-prop}.
\end{proof}

\begin{thm} \label{grothendieck-type-theorem}
 Let\/ $\sA$ be a finitely accessible category endowed with an exact
structure with exact functors of directed colimit.
 Then\/ $\sA$ is an exact category of Grothendieck type.
\end{thm}

\begin{proof}
 In view of Proposition~\ref{efficient-prop}, it remains to check
the last axiom~(GT4).
 Indeed, for any ring with many objects $\cS$, the exact category of
flat $\cS$\+modules $\Modrfl\cS$ is deconstructible in itself,
essentially by~\cite[Lemma~1 and Proposition~2]{BBE}.
 In view of Lemma~\ref{pure-exact-structure-is-flat-modules}, it follows
that any finitely accessible additive category $\sA$ endowed with
the pure exact structure is deconstructible in itself.
 As, under the assumptions of the theorem, any filtration in the pure
exact structure on $\sA$ is also a filtration in the given
exact structure on $\sA$, we can conclude that any exact category $\sA$
satisfying the assumptions of the theorem is deconstructible in itself.
\end{proof}

\begin{cor} \label{loc-coh-are-grothendieck-type-cor}
 All locally coherent exact categories are of Grothendieck type.
 In particular, in any locally coherent exact category, there are
enough injective objects.
\end{cor}

\begin{proof}
 Follows from Theorem~\ref{grothendieck-type-theorem}
and~\cite[Corollary~5.9]{Sto-ICRA}.
 Alternatively, one can refer to
Propositions~\ref{canonical-embedding-properties}(a,c)
and~\ref{locally-FP-infty-prop} below,
\cite[Proposition~A(i\+-ii) in Section~0.2]{Pgen},
and~\cite[Theorem~3.16]{Sto-ICRA} (for the first
assertion of the corollary).
\end{proof}

\Section{Canonical Embedding into Abelian Category}
\label{canonical-embedding-secn}

 The aim of this section is to represent an arbitrary locally
coherent exact category $\sC$ as a full subcategory in a locally
finitely presentable (in fact, locally type~$\FP_\infty)$ Grothendieck
abelian category~$\sK$, and show that $\sC$ enjoys nice closure
properties as a full subcategory in~$\sK$.
 It is important that the locally coherent exact structure on $\sC$
turns out to be inherited from the abelian exact structure on~$\sK$.

 For any small exact category $\sE$, there is a classical construction
of the canonical embedding of $\sE$ into the abelian category $\sK$ of
\emph{additive sheaves} on~$\sE$.
 Let us recall some details~\cite[Proposition~A.2]{Kel},
\cite[Section~A.7]{TT}, \cite[Appendix~A]{Bueh}.

 The category $\sK$ is defined as the category of all left exact
contravariant functors $\sE^\sop\rarrow\Ab$.
 Here $\Ab$ denotes the category of abelian groups; and a contravariant
functor $K\:\sE^\sop\rarrow\Ab$ is said to be \emph{left exact} if,
for any admissible short exact sequence $0\rarrow E'\rarrow E\rarrow
E''\rarrow0$ in $\sE$, the sequence of abelian groups $0\rarrow K(E'')
\rarrow K(E)\rarrow K(E')$ is left exact.
 The functor $\rho\:\sE\rarrow\sK$ assigns to every object $E\in\sE$
the representable functor $\rho(E)=\Hom_\sE({-},E)$.
 The following proposition summarizes the classical theory.

\begin{prop} \label{canonical-embedding-properties}
\textup{(a)} The functor~$\rho$ is fully faithful.
 The category\/ $\sK$ is abelian. \par
\textup{(b)} A short sequence\/ $0\rarrow E'\rarrow E\rarrow E''
\rarrow0$ in\/ $\sE$ is admissible exact if and only if its image\/
$0\rarrow\rho(E')\rarrow\rho(E)\rarrow\rho(E'')\rarrow0$ is exact
in~$\sK$. \par
\textup{(c)} The essential image of the functor~$\rho$ is closed under
extensions in\/~$\sK$. \par
\textup{(d)} For any objects $K\in\sK$ and $E\in\sE$, and any
epimorphism $K\rarrow\rho(E)$ in\/ $\sK$, there exists an admissible
epimorphism $F\rarrow E$ in\/ $\sE$ and a morphism $\rho(F)\rarrow K$
in\/ $\sK$ making the triangular diagram $\rho(F)\rarrow K\rarrow
\rho(E)$ commutative in\/~$\sK$. \par
\textup{(e)} If the category\/ $\sE$ is weakly idempotent-complete,
then the essential image of the functor~$\rho$ is closed under
kernels of epimorphisms in\/~$\sK$.
\end{prop}

\begin{proof}
 The first assertion of part~(a) holds by Yoneda's lemma.
 The second assertion of part~(a) is~\cite[Proposition~A.7.13]{TT}
or~\cite[Lemma~A.20]{Bueh}.
 Part~(b) is~\cite[Propositions~A.7.14 and~A.7.16(a)]{TT}
or~\cite[Lemmas~A.21 and~A.23]{Bueh}.
 Part~(c) is~\cite[Lemma~A.7.18]{TT} or~\cite[Lemma~A.24]{Bueh}.
 Part~(d) is~\cite[second paragraph of Lemma~A.7.15]{TT}
or~\cite[Lemma~A.22]{Bueh}.
 Part~(e) is~\cite[Proposition~A.7.16(b)]{TT}.
\end{proof}

 Following~\cite[Section~3]{Gil} and~\cite[Section~3 of the published
version or Section~2 of the \texttt{arXiv} version]{BGP}, an object $S$
in a Grothendieck category $\sK$ is said to be \emph{of type
$\FP_\infty$} if, for every $n\ge0$, the functor $\Ext^n_\sK(S,{-})\:
\sK\rarrow\Ab$ preserves directed colimits.
 Obviously, objects of type $\FP_\infty$ are finitely presentable.
 A Grothendieck category $\sK$ is said to be \emph{locally type
$\FP_\infty$} if it has a set of generators consisting of objects
of type~$\FP_\infty$.

\begin{prop} \label{locally-FP-infty-prop}
 The category\/ $\sK$ is a locally finitely presentable Grothendieck
category, and in fact, it is locally type\/~$\FP_\infty$.
 The objects in the image of the functor $\rho\:\sE\rarrow\sK$ are
of type\/ $\FP_\infty$, and they form a set of generators of\/~$\sK$.
\end{prop}

\begin{proof}
 By construction, $\sK$ is a full subcategory in the category
$\Fun_\boZ(\sE^\sop,\Ab)$ of all contravariant additive functors
$\sE^\sop\rarrow\Ab$.
 Specifically, the full subcategory $\sK\subset\Fun_\boZ(\sE^\sop,\Ab)$
consists of all the left exact functors.
 Since directed colimits of left exact sequences of abelian groups
are left exact, the full subcategory $\sK$ is closed under directed
colimits in $\Fun_\boZ(\sE^\sop,\Ab)$.
 The objects $\rho(E)$, \,$E\in\sE$ are finitely presentable (in fact,
finitely generated projective) in $\Fun_\boZ(\sE^\sop,\Ab)$, and it
follows that these objects are also finitely presentable in~$\sK$.

 The objects $\rho(E)$ form a set of generators of
$\Fun_\boZ(\sE^\sop,\Ab)$, hence they also form a set of generators
of~$\sK$.
 By~\cite[Theorem~1.11]{AR}, the category $\sK$ is locally finitely
presentable.
 The assertion that the category $\sK$ is Grothendieck can be found
in~\cite[Proposition~A.7.13]{TT}; but in fact any locally finitely
presentable abelian category is
Grothendieck~\cite[Proposition~1.59]{AR}.

 It remains to show that the objects $\rho(E)$ are of type $\FP_\infty$
in~$\sK$.
 Replacing the category $\sE$ with its (weak) idempotent completion
does not change the category $\sK$, so without loss of generality we
can assume $\sE$ to be weakly idempotent-complete.
 Then Proposition~\ref{canonical-embedding-properties}(e) is
applicable, and it remains to refer to~\cite[Lemma~3.5]{Pgen}.
\end{proof}

\begin{thm} \label{canonical-embedding-theorem}
 Let\/ $\sC$ be a locally coherent exact category and\/ $\sC_\fp\subset
\sC$ be its exact subcategory of finitely presentable objects.
 Put\/ $\sE=\sC_\fp$, and consider the canonical embedding $\rho\:\sE
\rarrow\sK$ of the exact category\/ $\sE$ into the abelian category\/
$\sK$ of left exact functors\/ $\sE^\sop\rarrow\Ab$.
 Then the full subcategory\/ $\varinjlim\rho(\sE)\subset\sK$ is
closed under coproducts, directed colimits, extensions, and kernels
of epimorphisms in the abelian category\/~$\sK$.
 Endowed with the exact category structure inherited from
the abelian category\/ $\sK$, the category\/ $\varinjlim\rho(\sE)$
is naturally equivalent to the exact category\/~$\sC$.
\end{thm}

\begin{proof}
 To establish the closure properties of the full subcategory
$\varinjlim\rho(\sE)\subset\sK$, it suffices to refer to
Proposition~\ref{canonical-embedding-properties}(c,e),
Proposition~\ref{locally-FP-infty-prop},
and~\cite[Proposition~3.6]{Pgen}.
 By Proposition~\ref{accessible-subcategory}, the category
$\varinjlim\rho(\sE)$ is finitely accessible and $\rho(\sE)$ is its
full subcategory of finitely presentable objects; so
$\varinjlim\rho(\sE)\simeq\Ind\rho(\sE))$.
 Similarly we have $\sC\simeq\Ind(\sE)$, so the category equivalence
$\rho\:\sE\simeq\rho(\sE)$ induces a category equivalence
$\sC\simeq\varinjlim\rho(\sE)$.
 It remains to show that the inherited exact structure on
$\varinjlim\rho(\sE)\subset\sK$ agrees with the original locally
coherent exact structure on~$\sC$.

 The (admissible) short exact sequences in $\sC$ are the directed
colimits of short exact sequences in~$\sE$.
 Since the functor~$\rho$ is exact by
Proposition~\ref{canonical-embedding-properties}(b) and the directed
colimits are exact functors in $\sK$, it follows that any short
exact sequence in $\sC$ is also a short exact sequence in
$\varinjlim\rho(\sE)$.
 To prove the converse implication, we will use the characterization
of admissible epimorphisms in $\sC$ provided by
Lemma~\ref{admissible-epis-in-the-colimit-exact-structure}.

 Let $0\rarrow L\rarrow M\rarrow N\rarrow0$ be a short exact sequence
in $\sK$ with the terms $L$, $M$, $N\in\varinjlim\rho(\sE)$.
 Given an object $E\in\rho(\sE)$ and a morphism $E\rarrow N$ in $\sK$,
we need to construct a commutative square
$$
 \xymatrix{
  M \ar[r] & N \\
  F \ar@{..>}[u] \ar@{..>>}[r] & E \ar[u]
 }
$$
with an object $F\in\rho(\sE)$ and an admissible epimorphism
$F\rarrow E$ in~$\rho(\sE)$.
 Indeed, let $K$ be the pullback of the pair of morphisms $M\rarrow N$
and $E\rarrow N$ in the abelian category~$\sK$.
 Then the morphism $K\rarrow E$ is an epimorphism in $\sK$, since
the morphism $M\rarrow N$ is.
 It remains to refer to
Proposition~\ref{canonical-embedding-properties}(d).

 Now that we have shown that the morphism $M\rarrow N$ comes from
an admissible epimorphism in~$\sC$, it follows that the whole
short exact sequence $0\rarrow L\rarrow M\rarrow N\rarrow0$ comes
from a short exact sequence in~$\sC$ (since we already know that
the functor $\sC\rarrow\varinjlim\rho(\sE)\subset\sK$ is exact).
\end{proof}

\Section{Two Periodicity Theorems}

 Let $\sC$ be an idempotent-complete additive category and
$\sF\subset\sC$ be a class of objects.
 Following the notation of Section~\ref{ind-objects-secn}, we denote
by $\overline\sF\subset\sC$ the class of all direct summands of
objects from $\sF$ in~$\sC$.

 The aim of this section is to formulate and prove two periodicity
theorems for a locally coherent exact category.
 We state the first one right away, postponing the formulation of
the second one towards the end of the section.

\begin{thm} \label{first-periodicity-theorem}
 Let\/ $\sC$ be a locally coherent exact category and\/
$\sP=\overline{\Fil(\sC_\fp)}$ be the class of all direct summands
of objects filtered by finitely presentable objects in\/~$\sC$
(in the sense of the definition in
Section~\ref{grothendieck-type-secn}).
 Let\/ $0\rarrow C\rarrow P\rarrow C\rarrow0$ be an admissible
short exact sequence in\/ $\sC$ with $P\in\sP$.
 Then one has $C\in\sP$. 
\end{thm}

\begin{exs} \label{first-periodicity-examples}
 (1)~Let $R$ be a ring and $\sC=\Modrfl R$ be the category of flat
right $R$\+modules.
 Then $\sC$ is a finitely accessible category, and the finitely
presentable objects of $\sC$ are the finitely generated projective
right $R$\+modules.
 Endowed with the exact structure inherited from the ambient abelian
category of modules $\Modr R$, the category $\sC$ becomes a locally
coherent exact category with the pure exact structure (see
Example~\ref{flat-modules-example}).

 In this context, the class $\sP=\overline{\Fil(\sC_\fp)}\subset\sC$
is the class of all projective $R$\+modules.
 Theorem~\ref{first-periodicity-theorem} claims that, for any short
exact sequence of $R$\+modules $0\rarrow C\rarrow P\rarrow C\rarrow0$
with a flat $R$\+module $C$ and a projective $R$\+module $P$,
the $R$\+module $C$ is also projective.
 This result is due to Benson and Goodearl~\cite[Theorem~2.5]{BG},
see also Neeman~\cite[Remark~2.15 and Theorem~8.6]{Neem}.

\smallskip
 (2)~Let $R$ be a ring and $\sC=\Modr R$ be the abelian category of
right $R$\+modules.
 Then $\sC$ is a locally finitely presentable category, and
the finitely presentable objects of $\sC$ are the cokernels of
morphisms of finitely generated free $R$\+modules.
 Endowed with the maximal locally coherent exact structure as per
Section~\ref{example-maximal-subsecn}, \,$\sC$ becomes a locally
coherent exact category.
 Assume additionally that the ring $R$ is right coherent; then
$\sC$ is a locally coherent abelian category, so the maximal
locally coherent exact structure on $\sC$ coincides with the abelian
exact structure by Corollary~\ref{locally-coherent-abelian-corollary}.

 In this context, the class $\sP=\overline{\Fil(\sC_\fp)}\subset\sC$
is known as the class of all \emph{fp\+projective} $R$\+modules.
 Theorem~\ref{first-periodicity-theorem} claims that, for any short
exact sequence of $R$\+modules $0\rarrow C\rarrow P\rarrow C\rarrow0$
with an fp\+projective $R$\+module $P$, the $R$\+module $C$ is also
fp\+projective.
 This result is due to \v Saroch and
\v St\!'ov\'\i\v cek~\cite[Example~4.3]{SarSt},
see also~\cite[Theorem~0.7 or Corollary~4.9]{BHP}.

 For an extension of example~(2) to arbitrary rings~$R$, see
Example~\ref{mlc-periodicity-module-example} below.
\end{exs}

 Let $\sK$ be an exact category, and let $\sA$, $\sB\subset\sK$ be two
classes of objects.
 Then $\sA^{\perp_1}\subset\sK$ denotes the class of all objects
$X\in\sK$ such that $\Ext_\sK^1(A,X)=0$ for all $A\in\sA$.
 Similarly, $\sA^{\perp_{\ge1}}\subset\sK$ is the class of all objects
$X\in\sK$ such that $\Ext_\sK^n(A,X)=0$ for all $A\in\sA$ and $n\ge1$.
 Dually, ${}^{\perp_1}\sB\subset\sK$ is the class of all objects
$Y\in\sK$ such that $\Ext_\sK^1(Y,B)=0$ for all $B\in\sB$, while
${}^{\perp_{\ge1}}\sB\subset\sK$ is the class of all objects
$Y\in\sK$ such that $\Ext_\sK^n(Y,B)=0$ for all $B\in\sB$ and $n\ge1$.

 The following classical result is known as the \emph{Eklof lemma}.

\begin{lem} \label{eklof-lemma}
 For any exact category\/ $\sK$ and any class of objects\/ $\sB
\subset\sK$, all objects filtered by objects from\/ ${}^{\perp_1}\sB
\subset\sK$ belong to\/ ${}^{\perp_1}\sB$, that is\/
$\Fil({}^{\perp_1}\sB)={}^{\perp_1}\sB$.
\end{lem}

\begin{proof}
 The argument from~\cite[Lemma~4.5]{PR} applies;
see also~\cite[Proposition~5.7]{Sto-ICRA} for an earlier approach
and~\cite[Lemma~1.1]{BHP} for further references.
\end{proof}

 A pair of classes of objects $(\sA,\sB)$ in an exact category $\sK$
is said to be a \emph{cotorsion pair} if $\sB=\sA^{\perp_1}$ and
$\sA={}^{\perp_1}\sB$.
 A cotorsion pair $(\sA,\sB)$ in $\sK$ is said to be \emph{complete}
if, for every object $K\in\sK$ there exist (admissible) short exact
sequences
\begin{gather*}
 0\lrarrow B'\lrarrow A\lrarrow K\lrarrow0 \\
 0\lrarrow K\lrarrow B\lrarrow A'\lrarrow0
\end{gather*}
in $\sK$ with $A$, $A'\in\sA$ and $B$, $B'\in\sB$.

 A class of objects $\sA\subset\sK$ is said to be \emph{generating}
if every object $K\in\sK$ is the codomain of an admissible epimorphism
$A\rarrow K$ with $A\in\sA$.
 Dually, a class of objects $\sB\subset\sK$ is said to be
\emph{cogenerating} if every object $K\in\sK$ is the domain of
an admissible monomorphism $K\rarrow B$ with $B\in\sB$.
 Notice that in any complete cotorsion pair $(\sA,\sB)$ the class
$\sA$ is generating and the class $\sB$ is cogenerating.
 The following lemma is standard.

\begin{lem} \label{hereditary-lemma}
 Let $(\sA,\sB)\subset\sK$ be a cotorsion pair such that the class
$\sA$ is generating and the class $\sB$ is cogenerating.
 Then the following conditions are equivalent:
\begin{enumerate}
\item the class\/ $\sA$ is closed under kernels of admissible
epimorphisms in\/~$\sK$;
\item the class\/ $\sB$ is closed under cokernels of admissible
monomorphisms in\/~$\sK$;
\item $\Ext^2_\sK(A,B)=0$ for all $A\in\sA$ and $B\in\sB$;
\item $\Ext^n_\sK(A,B)=0$ for all $A\in\sA$, \,$B\in\sB$, and $n\ge1$.
\end{enumerate}
\end{lem}

\begin{proof}
 See, e.~g., \cite[Lemma~6.17]{Sto-ICRA}.
\end{proof}

 A cotorsion pair $(\sA,\sB)$ satisfying the equivalent conditions
of Lemma~\ref{hereditary-lemma} is said to be \emph{hereditary}.

\begin{proof}[Proof of Theorem~\ref{first-periodicity-theorem}]
 Consider the exact category $\sE=\sC_\fp$, and let $\rho\:\sE
\rarrow\sK$ be its canonical embedding into an abelian category~$\sK$.
 According to Proposition~\ref{locally-FP-infty-prop}, \,$\sK$ is
a locally finitely presentable Grothendieck category, and the full
subcategory $\rho(\sE)\subset\sK$ consists of finitely presentable
objects.
 Following Theorem~\ref{canonical-embedding-theorem}, \,$\sC\simeq
\varinjlim\rho(\sE)$ is naturally a full subcategory in $\sK$,
and the locally coherent exact structure on $\sC$ is inherited from
the abelian exact structure on~$\sK$.

 We will obtain the assertion of
Theorem~\ref{first-periodicity-theorem} as a (rather special)
particular case of~\cite[Theorem~A in Section~0.2 or
Corollary~7.3]{Pgen} for $\sS=\sE=\sC_\fp\subset\sK$ and
$\kappa=\aleph_0$.
 We have $\sC=\varinjlim\sS\subset\sK$, and the class $\sC$ is
deconstructible in $\sK$ by~\cite[Proposition~A(i\+-ii)]{Pgen} together
with Proposition~\ref{locally-FP-infty-prop} above, or alternatively,
because $\sC$ is deconstructible in itself by
Corollary~\ref{loc-coh-are-grothendieck-type-cor} and closed under
directed colimits and extensions in $\sK$ by
Theorem~\ref{canonical-embedding-theorem}.
 Comparing the notation in Theorem~\ref{first-periodicity-theorem}
with the one in~\cite[Corollary~7.3]{Pgen}, we have $\sA=\sP$.

 Following~\cite[Theorem~5.16]{Sto-ICRA} or~\cite[Theorem~4.3]{Pgen},
we have a complete cotorsion pair $(\sA,\sB)$ in~$\sC$.
 By~\cite[Lemma~6.1]{PS6} and
Proposition~\ref{canonical-embedding-properties}(d,e), the cotorsion
pair $(\sA,\sB)$ is hereditary in~$\sC$.
 Alternatively, one can construct a complete cotorsion pair
$(\sA,\sA^{\perp_1})$ in $\sK$, observe that it is hereditary
by~\cite[Lemma~1.3]{BHP} with 
Proposition~\ref{canonical-embedding-properties}(d,e), and restrict it
to a hereditary complete cotorsion pair $(\sA,\>\sB=
\sA^{\perp_1}\cap\sC)$ in $\sC$ using~\cite[Lemma~2.2]{Pgen}.
 Notice that $\sA^{\perp_1}=\sS^{\perp_1}$ by Lemma~\ref{eklof-lemma},
and that the class $\sC$ is closed under the kernels of epimorphisms
in $\sK$ by Theorem~\ref{canonical-embedding-theorem}.

 The assertion that $(\sA,\sB)$ is a hereditary cotorsion pair in $\sC$
implies $\sA'=\sA$ and $\sB'=\sB$ in the notation
of~\cite[Corollary~7.3]{Pgen}.
 Applying~\cite[Corollary~7.3]{Pgen}, we obtain the desired assertion
of Theorem~\ref{first-periodicity-theorem}.
\end{proof}

\begin{rem}
 One can observe that the main argument in the proof
of~\cite[proof of Theorem~A or Theorem~7.1]{Pgen} happens
\emph{within} the exact category $\sC$, so it may appear that
the construction of the embedding into a locally finitely presentable
abelian category $\sK$ is unnecessary for the proof above.
 The problem is, however, that the Hill lemma~\cite[Theorem~2.1]{Sto0}
for the countable cardinal $\kappa=\aleph_0$ is only known for
locally finitely presentable Grothendieck categories, and \emph{not}
for exact categories of Grothendieck type.
 The Hill lemma, restated as~\cite[Proposition~2.6 or~6.7]{Pgen},
plays a key role in the proof of~\cite[Theorem~2.9(b) or~7.1(b)]{Pgen}.
 For this reason, the results of Section~\ref{canonical-embedding-secn}
are important for our proof of Theorem~\ref{first-periodicity-theorem}.
\end{rem}

 Here is our second periodicity theorem.

\begin{thm} \label{second-periodicity-theorem}
 Let\/ $\sC$ be a locally coherent exact category and\/ $\sC\rarrow\sK$
be its natural embedding into a locally finitely presentable abelian
category described in Theorem~\ref{canonical-embedding-theorem}.
 Put\/ $\sB=(\sC_\fp)^{\perp_1}\subset\sK$ and\/ $\sD=\sC^{\perp_1}
\subset\sK$.
 Let\/ $0\rarrow B\rarrow D\rarrow B\rarrow0$ be a short exact
sequence in\/ $\sK$ with $B\in\sB$ and $D\in\sD$.
 Then one has\/ $B\in\sD$.
\end{thm}

\begin{proof}
 This is a (rather special) particular case of~\cite[Theorem~B
in Section~0.3 or Theorem~1.2]{Pgen} for $\sS=\sC_\fp\subset\sK$
and $\sT=\varnothing$.
 Notice that the class $\sS=\sE=\sC_\fp$ is closed under extensions
and kernels of epimorphisms in $\sK$ by
Proposition~\ref{canonical-embedding-properties}(c,e),
so~\cite[Proposition~B]{Pgen} is applicable.
\end{proof}

\begin{ex} \label{second-periodicity-example}
 Let $R$ be a ring and $\sC=\Modrfl R$ be the category of flat
right $R$\+modules endowed with the pure exact structure, as in
Example~\ref{first-periodicity-examples}(1).
 Then $\sK=\Modr R$ is the abelian category of right $R$\+modules,
and the embedding $\sC\rarrow\sK$ from
Theorem~\ref{canonical-embedding-theorem} agrees with the identity
inclusion $\Modrfl R\rarrow\Modr R$.
 The full subcategory $\sB=(\sC_\fp)^{\perp_1}\subset\sK$ coincides
with the whole abelian module category, $\sB=\sK=\Modr R$, while
the full subcategory $\sD=\sC^{\perp_1}\subset\sK$ is the class of
all \emph{cotorsion} right $R$\+modules, $\sD=\Modrcot R$.

 Theorem~\ref{second-periodicity-theorem} claims that, for any
short exact sequence of $R$\+modules $0\rarrow B\rarrow D\rarrow B
\rarrow0$ with a cotorsion $R$\+module $D$, the $R$\+module $B$ is
also cotorsion.
 This result is due to Bazzoni, Cort\'es-Izurdiaga,
and Estrada~\cite[Theorem~1.2(2) or Proposition~4.8(2)]{BCE}.
\end{ex}

\Section{Maximal Locally Coherent Exact Structure~II}

 In Example~\ref{first-periodicity-examples}(2) we spelled out
what Theorem~\ref{first-periodicity-theorem} says in the context
of the abelian category $\sC=\Modr R$ of right modules over
a right coherent ring~$R$ (and the abelian exact structure on~$\sC$).
 The aim of this section is to extend this discussion to arbitrary
rings $R$ and the maximal locally coherent exact structure on
$\sC=\Modr R$.
 This will lead us to a new version of fp\+projective periodicity
theorem for modules over an arbitary ring, and more generally,
for objects of an arbitrary locally finitely presentable abelian
category.

 Let $\sC$ be a locally $\kappa$\+presentable Grothendieck category
(where $\kappa$~is a regular cardinal).
 Following~\cite[Remark~4.11]{BHP}, we will say that an object
$P\in\sC$ is \emph{$\kappa$\+p-projective} if $P$ is a direct summand
of an object filtered by $\kappa$\+presentable objects in~$\sC$. 
 In the case of the cardinal $\kappa=\aleph_0$, the term
\emph{fp\+projective objects} is used~\cite[Section~2]{BHP}.

\begin{lem} \label{fp-projective-kernel-lemma}
 Let\/ $\sC$ be a locally $\kappa$\+presentable Grothendieck category
and\/ $0\rarrow P\rarrow D\rarrow E\rarrow0$ be a short exact sequence
in\/ $\sC$ with a $\kappa$\+p-projective object~$P$.
 Then\/ $0\rarrow P\rarrow D\rarrow E\rarrow0$ is an admissible short
exact sequence in the maximal locally $\kappa$\+coherent exact
structure on\/~$\sC$.
\end{lem}

\begin{proof}
 Let $Q\in\sC$ be an object that the object $P\oplus Q$ is filtered
by $\kappa$\+presentable ones.
 Consider the short sequence $0\rarrow P\oplus Q\rarrow D\oplus Q
\rarrow E\rarrow0$ constructed as the direct sum of the short
sequences $0\rarrow P\rarrow D\rarrow E\rarrow0$ and $0\rarrow Q
\rarrow Q\rarrow0\rarrow0$.
 If the sequence $0\rarrow P\oplus Q\rarrow D\oplus Q\rarrow E\rarrow0$
is admissible exact, then so is the sequence $0\rarrow P\rarrow D
\rarrow E\rarrow0$ (by the pushout axiom).
 Hence we can assume without loss of generality that the object $P$
is filtered by $\kappa$\+presentable objects.

 It suffices to show that the morphism $D\rarrow E$ is an admissible
epimorphism in the maximal locally $\kappa$\+coherent exact structure
on~$\sC$.
 For this purpose, we will apply the criterion provided by
Corollary~\ref{maximal-admissible-epimorphisms-cor}.
 Given a $\kappa$\+presentable object $S$ and a morphism $S\rarrow E$
in~$\sC$, we need to construct a commutative square diagram
$$
 \xymatrix{
  0 \ar[r] & P \ar@{>->}[r] & D \ar@{->>}[r] & E \ar[r] & 0 \\
  && T \ar@{..>}[u] \ar@{..>>}[r] & S \ar[u]
 }
$$
with an epimorphism $T\rarrow S$ such that both $T$ and
$\ker(T\twoheadrightarrow S)$ are $\kappa$\+presentable objects
in~$\sC$.

 Denote by $H$ the pullback of the pair of morphisms $D\rarrow E$
and $S\rarrow E$ in~$\sC$.
 Then we have a short exact sequence $0\rarrow P\rarrow H
\rarrow S\rarrow0$ in~$\sC$.
 A given filtration of the object $P$ by $\kappa$\+presentable
objects then can be extended to such a filtration of the object $H$
by adding a single top quotient object~$S$ (since the object $S$ is
$\kappa$\+presentable).
 Applying the Hill lemma~\cite[Theorem~2.1]{Sto0}, the desired
$\kappa$\+presentable object $T$ can be now found as
a suitable subobject in~$H$.
\end{proof}

\begin{cor} \label{fp-projective-filtrations-cor}
 Let\/ $\sC$ be a locally $\kappa$\+presentable Grothendieck category.
 Then the class of all objects of\/ $\sC$ filtered by
$\kappa$\+presentable objects in the abelian exact structure on\/ $\sC$
coincides with the class of all objects of\/ $\sC$ filtered by
$\kappa$\+presentable objects in the maximal locally $\kappa$\+coherent
exact structure on\/~$\sC$.
\end{cor}

\begin{proof}
 It follows from Lemma~\ref{fp-projective-kernel-lemma} that
any filtration by $\kappa$\+presentable objects in the abelian exact
structure on $\sC$ is also a filtration in the maximal locally
$\kappa$\+coherent exact structure on~$\sC$.
\end{proof}

\begin{thm} \label{mlc-periodicity-theorem}
 Let\/ $\sC$ be a locally finitely presentable abelian category, and
let\/ $0\rarrow C\rarrow P\rarrow C\rarrow0$ be an admissible short
exact sequence in the maximal locally coherent exact structure
on\/~$\sC$.
 Assume that the object $P$ is fp\+projective.
 Then the object $C$ is fp\+projective as well.
\end{thm}

\begin{proof}
 Apply Theorem~\ref{first-periodicity-theorem} to the maximal locally
coherent exact structure on $\sC$, and take
Corollary~\ref{fp-projective-filtrations-cor} (for $\kappa=\aleph_0$)
into account.
\end{proof}

\begin{ex} \label{mlc-periodicity-module-example}
 In particular, let $R$ be an arbitrary ring and $0\rarrow C\rarrow P
\rarrow C\rarrow0$ be an admissible short exact sequence in the maximal
locally coherent exact structure on the category $\sC=\Modr R$.
 Theorem~\ref{mlc-periodicity-theorem} tells us that if the $R$\+module
$P$ is fp\+projective, then so is the $R$\+module~$C$.
\end{ex}

 Notice that the assertions of Theorem~\ref{mlc-periodicity-theorem}
and Example~\ref{mlc-periodicity-module-example} are certainly
\emph{not} true with the assumption of admissible exactness in
the maximal locally coherent exact structure replaced by exactness
in the abelian category~$\sC$ (unless one assumes local coherence,
which makes the two exact structures coincide).
 This is explained in~\cite[Corollary~4.6(2)\,$\Rightarrow$\,(4)]{BHP}.

\begin{rem}
 The proof of Theorem~\ref{mlc-periodicity-theorem} via the proof
of Theorem~\ref{first-periodicity-theorem}, applied in the context
of Example~\ref{mlc-periodicity-module-example}, may appear to be
confusing, so let us provide some explanation.
 Given the module category $\sC=\Modr R$ with its maximal locally
coherent exact structure, we consider the full subcategory
$\sE=\Modrfp R\subset\Modr R$ of finitely presentable $R$\+modules,
endowed with the maximal exact structure, which coincides with
the exact structure inherited from the abelian exact structure of
$\Modr R$ (as per Lemma~\ref{maximal-on-kappa-presentables}).
 Then we consider the abelian category $\sK$ of all left exact
functors $\sE^\sop\rarrow\Ab$.
 What category is that?

 The category of contravariant functors $(\Modrfp R)^\sop\rarrow\Ab$
taking cokernels to kernels is naturally equivalent to $\sC=\Modr R$,
as one can easily see.
 Any functor $(\Modrfp R)^\sop\rarrow\Ab$ taking cokernels to kernels
has the form $S\longmapsto\Hom_R(S,M)$, where $M\in\Modr R$.
 So what is the difference between the categories $\sC$ and~$\sK$\,?
 The difference is that a functor $K\:(\Modrfp R)^\sop\rarrow\Ab$
belonging to the category $\sK$ \emph{need not} take cokernels to
kernels.
 The left exactness condition imposed on objects of the category $\sK$
is weaker, and only requires the functor $K$ to take cokernels
\emph{of injective morphisms of finitely presentable $R$\+modules}
to kernels in~$\Ab$.
 The difference manifests itself when the ring $R$ is not right
coherent.
\end{rem}

\appendix

\bigskip
\section*{Appendix.  Generalities on Accessible Categories}
\medskip
\setcounter{section}{1}
\setcounter{thm}{0}

 Throughout this appendix, $\kappa$~is a regular cardinal.
 We use the book~\cite{AR} as the background reference source on
accessible categories.
 In particular, we refer to~\cite[Definition~1.4, Theorem~1.5,
Corollary~1.5, Definition~1.13(1), and Remark~1.21]{AR} for
an important discussion of \emph{$\kappa$\+directed posets}
vs.\ \emph{$\kappa$\+filtered small categories}, and accordingly,
$\kappa$\+directed vs.\ $\kappa$\+filtered colimits (for a regular
cardinal~$\kappa$).

 Let $\sA$ be a category with $\kappa$\+directed (equivalently,
$\kappa$\+filtered) colimits.
 An object $S\in\sA$ is said to be \emph{$\kappa$\+presentable}
if the functor $\Hom_\sA(S,{-})\:\sA\rarrow\Sets$ preserves
$\kappa$\+directed colimits.
 We will use the notation $\sA_{<\kappa}\subset\sA$ for the full
subcategory of all $\kappa$\+presentable objects of~$\sA$.
 Given a class of objects $\sS\subset\sA$, we denote by
$\varinjlim_{(\kappa)}\sS\subset\sA$ the class of all
$\kappa$\+directed colimits of objects from $\sS$ in~$\sA$.
 In the case of the cardinal $\kappa=\aleph_0$, we will use
the notation $\varinjlim\sS$ instead of $\varinjlim_{(\aleph_0)}\sS$.

 The category $\sA$ is said to be
\emph{$\kappa$\+accessible}~\cite[Definition~2.1]{AR} if
there is a set of $\kappa$\+presentable objects $\sS\subset\sA$
such that every object of $\sA$ is a $\kappa$\+directed colimit
of objects from~$\sS$, i.~e., $\sA=\varinjlim_{(\kappa)}\sS$.
 If this is the case, then the $\kappa$\+presentable objects of $\sA$
are precisely all the retracts of the objects from~$\sS$.

 In particular, $\aleph_0$\+presentable objects are known as 
\emph{finitely presentable}~\cite[Definition~1.1]{AR}, and
$\aleph_0$\+accessible categories are called \emph{finitely
accessible}~\cite[Remark~2.2(1)]{AR}.

 The following proposition is well-known.

\begin{prop} \label{accessible-subcategory}
 Let\/ $\sA$ be a $\kappa$\+accessible category and\/
$\sT\subset\sA_{<\kappa}$ be a set of $\kappa$\+presentable objects.
 Then the class of objects\/ $\varinjlim_{(\kappa)}\sT\subset\sA$
is closed under $\kappa$\+directed colimits in\/~$\sA$.
 An object $B\in\sA$ belongs to\/ $\varinjlim_{(\kappa)}\sT$ if and
only if, for every object $S\in\sA_{<\kappa}$, any morphism $S\rarrow B$
in\/ $\sA$ factorizes through an object from\/~$\sT$.

 The full subcategory\/ $\varinjlim_{(\kappa)}\sT\subset\sA$ is
$\kappa$\+accessible, and its $\kappa$\+presentable objects are
precisely those objects of\/ $\varinjlim_{(\kappa)}\sT$ that are
$\kappa$\+presentable in\/~$\sA$.
 The intersection\/ $\sA_{<\kappa}\cap\varinjlim\sT$ consists precisely
of all the retracts of objects from\/~$\sT$ in\/~$\sA$.

 Assume that the category\/ $\sA$ is additive and the set of objects\/
$\sT$ is closed under finite direct sums in\/~$\sA$.
 Then so is the class of objects\/ $\varinjlim_{(\kappa)}\sT$.
 When\/ $\kappa=\aleph_0$, the class of objects\/
$\varinjlim_{(\kappa)}\sT$ is closed under all coproducts in\/~$\sA$.
\end{prop}

\begin{proof}
 In the context of additive categories and $\kappa=\aleph_0$, this
result, going back to~\cite[Proposition~2.1]{Len}, can be found
in~\cite[Section~4.1]{CB} and~\cite[Proposition~5.11]{Kra}.
 (The terminology ``locally finitely presented categories'' was used
in~\cite{CB,Kra} for what we call finitely accessible categories.)

 The nontrivial part is to prove the ``if'' implication in the second
assertion: if any morphism $S\rarrow B$ factorizes through an object
from $\sT$, then $B\in\varinjlim_{(\kappa)}\sT$.
 Here one has to use the assumption that $B$ is a $\kappa$\+directed
colimit of objects from $\sA_{<\kappa}$ in $\sA$, and show that
the canonical diagram of morphisms into $B$ from objects of $\sT$ is
$\kappa$\+filtered and has $B$ as the colimit.
 We refer to~\cite[Proposition~1.2]{Pacc} for further details.
\end{proof}

\subsection{Comma-categories with equations}
\label{appendix-comma-equations-subsecn}
 We start with the following simple lemma, which will be used
in Section~\ref{appendix-diagram-categories-subsecn}.

\begin{lem} \label{product-accessible}
 Let $\kappa$ be a regular cardinal, $I$ be a set of the cardinality
smaller than~$\kappa$, and $(\sK_i)_{i\in I}$ be an $I$\+indexed
family of $\kappa$\+accessible categories.
 Then the Cartesian product\/ $\sK=\prod_{i\in I}\sK_i$ is
a $\kappa$\+accessible category again.
 The $\kappa$\+presentable objects of\/ $\sK$ are precisely all
the collections of objects $(S_i\in\sK_i)_{i\in I}$ where the object
$S_i$ is $\kappa$\+presentable in\/ $\sK_i$ for every $i\in I$.
\end{lem} 

\begin{proof}
 This assertion can be found~\cite[proof of Proposition~2.67]{AR},
with the difference that the assumption that the cardinality of $I$
is smaller than~$\kappa$ is erroneously missing in~\cite{AR}.
 We refer to~\cite[Proposition~2.1]{Pacc} for the details.
\end{proof}

 Let $\sK$, $\sL$, and $\sM$ be three categories, and let
$F\:\sK\rarrow\sM$ and $G\:\sL\rarrow\sM$ be two functors.
 In this context, the \emph{comma-category} $\sC=F\downarrow\nobreak G$
is the category of triples $(K,L,\theta)$, where $K\in\sK$ and $L\in\sL$
are objects, while $\theta\:F(K)\rarrow G(L)$ is a morphism in~$\sM$.

 It is easy to see that if $\kappa$\+directed colimits exist in $\sK$,
$\sL$, and $\sM$, and are preserved by the functor $F$, then
$\kappa$\+directed colimits also exist in $\sC$, and are preserved by
the natural forgetful functors $\sC\rarrow\sK$ and $\sC\rarrow\sL$.

\begin{prop} \label{comma-accessible}
 Assume that the categories $\sK$, $\sL$, and\/ $\sM$ are
$\kappa$\+accessible, and the functors $F$ and $G$ preserve
$\kappa$\+directed colimits.
 Assume further that the functor $F$ takes $\kappa$\+presentable
objects to $\kappa$\+presentable objects.
 Then the category\/ $\sC$ is also $\kappa$\+accessible, and
the $\kappa$\+presentable objects of\/ $\sC$ are precisely all
the triples $(S,T,\sigma)\in\sC$ where the object $S$ is
$\kappa$\+presentable in\/ $\sK$ and the object $T$ is
$\kappa$\+presentable in\/~$\sL$.
\end{prop}

\begin{proof}
 This is~\cite[proof of Theorem~2.43]{AR}.
 Notice that~\cite[Theorem~4.1]{Pacc} is \emph{not} applicable here,
as the latter theorem requires two cardinals $\lambda<\kappa$, which
we do not have (and do not need) in the context of this proposition.
\end{proof}

 Consider the comma-category $\sC=F\downarrow G$ as above, and
denote by $\Pi_\sK\:\sC\rarrow\sK$ and $\Pi_\sL\:\sC\rarrow\sL$
the two natural forgetful functors.
 Assume further that we are given a category $\sN$ and two functors
$P\:\sK\rarrow\sN$ and $Q\:\sL\rarrow\sN$.
 Consider the two compositions of functors $P\Pi_\sK\:\sC\rarrow\sN$
and $Q\Pi_\sL\:\sC\rarrow\sN$, and suppose given a pair of parallel
natural transformations $\phi$, $\psi\:P\Pi_\sK
\rightrightarrows Q\Pi_\sL$.

 Denote by $\sE\subset\sC$ the full subcategory consisting of all
the objects $C\in\sC$ such that $\phi_C=\psi_C$.
 We will call the category $\sE$ the \emph{comma-category with
equations}.

 Clearly, if $\kappa$\+directed colimits exist in the four categories
$\sK$, $\sL$, $\sM$, and $\sN$, and are preserved by the two functors
$F$ and $P$, then the full subcategory $\sE$ is closed under
$\kappa$\+directed colimits in~$\sC$.

\begin{thm} \label{comma-category-with-equations-theorem}
 Assume that the four categories\/ $\sK$, $\sL$, $\sM$, and\/ $\sN$ are
$\kappa$\+accessible, and the four functors $F$, $G$, $P$, $Q$ preserve
$\kappa$\+directed colimits.
 Assume further that the functors $F$ and $P$ take
$\kappa$\+presentable objects to $\kappa$\+presentable objects.
 Then the category\/ $\sE$ is also $\kappa$\+accessible, and
the $\kappa$\+presentable objects of\/ $\sE$ are precisely all
the triples $(U,V,\upsilon)\in\sE\subset\sC$ where the object $U$ is
$\kappa$\+presentable in\/ $\sK$ and the object $V$ is
$\kappa$\+presentable in\/~$\sL$.
\end{thm}

\begin{proof}
 In view of Propositions~\ref{accessible-subcategory}
and~\ref{comma-accessible}, it suffices to check that, for every
pair of objects $(K,L,\theta)\in\sE$ and $(S,T,\sigma)\in
\sC_{<\kappa}$, and any morphism $(f,g)\:(S,T,\sigma)\rarrow
(K,L,\theta)$ in\/ $\sC$, there exists an object
$(U,V,\upsilon)\in\sC_{<\kappa}\cap\sE$ such that the morphism $(f,g)$
factorizes through the object $(U,V,\upsilon)$ in the category~$\sC$.

 Indeed, let $L=\varinjlim_{\xi\in\Xi}V_\xi$ be a representation of
the object $L\in\sL$ as a $\kappa$\+directed colimit of
$\kappa$\+presentable objects $V_\xi$, indexed by some
$\kappa$\+directed poset~$\Xi$.
 Then, since the object $T\in\sL$ is $\kappa$\+presentable,
there exists an index $\xi_0\in\Xi$ such that the morphism
$g\:T\rarrow L$ factorizes through the morphism $V_\xi\rarrow L$.

 Since $(K,L,\theta)\in\sE$ and $(f,g)$ is a morphism in~$\sC$,
the two compositions
$$
 \xymatrix{
  P(S) \ar@<2pt>[rr]^{\phi_{(S,T,\sigma)}}
  \ar@<-2pt>[rr]_{\psi_{(S,T,\sigma)}}
  && Q(T) \ar[r] & Q(V_{\xi_0}) \ar[r] & Q(L)
 }
$$
are equal to each other in the category~$\sN$.
 As $Q(L)=\varinjlim_{\xi\in\Xi}Q(V_\xi)$ in $\sN$ and
$P(S)$ is a $\kappa$\+presentable object in $\sN$, it follows that
there exists an index $\xi_1\in\Xi$, \,$\xi_1\ge\xi_0$, such that
the two compositions
$$
 \xymatrix{
  P(S) \ar@<2pt>[rr]^{\phi_{(S,T,\sigma)}}
  \ar@<-2pt>[rr]_{\psi_{(S,T,\sigma)}}
  && Q(T) \ar[r] & Q(V_{\xi_0}) \ar[r] & Q(V_{\xi_1})
 }
$$
are equal to each other in~$\sN$.

 It remains to put $U=S$ and $V=V_{\xi_1}$, and let $\upsilon\:F(U)
\rarrow G(V)$ be the composition $F(S)\overset\sigma\rarrow G(T)
\rarrow G(V_{\xi_1})$ in the category~$\sM$.
\end{proof}

\subsection{Finitely presented rigid diagram categories}
\label{appendix-diagram-categories-subsecn}
 We start with restating a result going back to~\cite[Expos\'e~I]{SGA4}.
 Given a small category $\sD$ and a category $\sA$, we denote by
$\Fun(\sD,\sA)$ the category of functors $\sD\rarrow\sA$.

\begin{prop} \label{nonadditive-functors-prop}
 Let\/ $\sA$ be a $\kappa$\+accessible category and\/ $\sD$ be
a finite category with no nonidentity endomorphisms of objects.
 Then the category\/ $\Fun(\sD,\sA)$ is $\kappa$\+accessible.
 The $\kappa$\+presentable objects of\/ $\Fun(\sD,\sA)$ are
the functors\/ $\sD\rarrow\sA_{<\kappa}$.
\end{prop}

\begin{proof}
 In the case of $\kappa=\aleph_0$, this is the result
of~\cite[Expos\'e~I, Proposition~8.8.5]{SGA4} or~\cite[page~55]{Mey}.
 For an arbitrary regular cardinal~$\kappa$, the assertion of
the proposition is a particular case of~\cite[Theorem~1.3]{Hen}.
\end{proof}

 The aim of this section is to extend the result of
Proposition~\ref{nonadditive-functors-prop} from arbitrary nonadditive
to \emph{$k$\+linear} functors, where $k$~is a commutative ring.
 The related assertion is stated as
Proposition~\ref{k-linear-functors-prop} below.
 Both Propositions~\ref{nonadditive-functors-prop}
and~\ref{k-linear-functors-prop} are provable by induction on
the number of objects in the category $\sD$ using
Theorem~\ref{comma-category-with-equations-theorem} for
the induction step.
 We will spell out the proof of
Proposition~\ref{k-linear-functors-prop}, as this is the result
we use in the main body of this paper.

 A \emph{$k$\+linear category} $\sA$ is a category enriched in
$k$\+modules.
 So, for any two objects $A$ and $B\in\sA$, the set $\Hom_\sA(A,B)$
has a $k$\+module structure, and the composition maps are
$k$\+bilinear.
 Let us describe the class of small $k$\+linear categories $\sD$ to
which our result applies.

 Suppose given a totally ordered, finite set of objects~$\{a\}$.
 For every pair of objects $a<b$, suppose given a \emph{generating
set} of morphisms $\Gen(a,b)$.
 For $a\ge b$, put $\Gen(a,b)=\varnothing$.
 Then, just as in~\cite[Section~6]{Pacc}, one can construct
the $k$\+linear category $\sB$ on the given set of objects
\emph{freely generated} by the given sets of morphisms.
 For every pair of objects~$a$, $b$, the free $k$\+module
$\Hom_\sB(a,b)$ is spanned by the set of all formal compositions
$g_n\dotsm g_1$, \,$n\ge0$, where $g_i\in\Gen(c_i,c_{i+1})$,
\ $c_1=a$, \,$c_{n+1}=b$.
 Notice that one has $\Hom_\sB(a,a)=k$ for all objects~$a$,
and $\Hom_\sB(a,b)=0$ if $a>b$.

 Furthermore, suppose given a \emph{set of defining relations}
$\Rel(a,b)\subset\Hom_\sB(a,b)$ for every pair of objects $a<b$.
 Let $\sJ$ be the two-sided ideal of morphisms in $\sB$ generated
by all the sets $\Rel(a,b)$.
 Consider the $k$\+linear quotient category $\sD=\sB/\sJ$ of
the $k$\+linear category $\sA$ by the ideal of morphisms~$\sJ$.

 We will say that a $k$\+linear category $\sD$ is a \emph{finitely
presented rigid category} if it has the form $\sD=\sB/\sJ$ as per
the construction above, where the set of all objects~$a$,
the set of all generators $\coprod_{a,b}\Gen(a,b)$, and the set
of all relations $\coprod_{a,b}\Rel(a,b)$ are all finite.
 The word ``rigid'' here refers to the assumption that all
the generating morphisms go in one direction with respect to
a given total order on the set of objects.

 Given a small $k$\+linear category $\sD$ and a $k$\+linear
category $\sA$, we denote by $\Fun_k(\sD,\sA)$ the $k$\+linear
category of $k$\+linear functors $\sD\rarrow\sA$.
 The following proposition is a $k$\+linear version of
Proposition~\ref{nonadditive-functors-prop}.

\begin{prop} \label{k-linear-functors-prop}
 Let\/ $\sA$ be a $\kappa$\+accessible $k$\+linear category and\/
$\sD$ be a finitely presented rigid $k$\+linear category.
 Then the category\/ $\Fun_k(\sD,\sA)$ is $\kappa$\+accessible.
 The $\kappa$\+presentable objects of\/ $\Fun_k(\sD,\sA)$ are
the $k$\+linear functors\/ $\sD\rarrow\sA_{<\kappa}$.
\end{prop}

\begin{proof}
 The argument proceeds by induction on the number of objects in~$\sD$.
 If the set of objects of $\sD$ is a singleton, then one has
$\Fun_k(\sD,\sA)\simeq\sA$, and there is nothing to prove.
 Otherwise, we separate all the objects of $\sD$ into two nonempty
groups, the upper and the lower one, so that one has $a'<a''$ for
any object~$a'$ from the lower group and any object~$a''$ from
the upper one.
 Let $\sD'$ and $\sD''\subset\sD$ be the full subcategories on
the lower and the upper groups of objects in $\sD$, respectively.
 Then both $\sD'$ and $\sD''$ are also finitely presented rigid
categories; indeed, the sets of generators and relations of $\sD'$
are just the sets of all generators and relations of $\sD$ on objects
from $\sD'$, and similarly for~$\sD''$.

 The claim is that the category $\sE=\Fun(\sD,\sA)$ can be produced
from the categories $\sK=\Fun(\sD',\sA)$ and $\sL=\Fun(\sD'',\sA)$
by the construction of the comma-category with equations from
Section~\ref{appendix-comma-equations-subsecn}.
 Specifically, let $\sM$ be the product of copies of the category $\sA$
indexed by the finite set $\coprod_{a'\in\sD',a''\in\sD''}\Gen(a',a'')$.
 Let $F\:\sK\rarrow\sM$ be the functor assigning to a functor
$A'\:\sD'\rarrow\sA$ the collection of objects whose component indexed
by a generator $g\in\Gen(a',a'')$ is the object $A'(a')\in\sA$.
 Similarly, let $G\:\sL\rarrow\sM$ be the functor assigning to a functor
$A''\:\sD''\rarrow\sA$ the collection of objects whose component indexed
by a generator $g\in\Gen(a',a'')$ is the object $A''(a'')\in\sA$.
 Then the comma-category $\sC=F\downarrow G$ is naturally equivalent
to the category of functors $\Fun_k(\widetilde\sB,\sA)$, where
$\widetilde\sB$ is the finitely presented rigid category whose objects
are all the objects of $\sD$, generating morphisms are all
the generating morphisms of $\sD$, and the set of defining relations is
$\coprod_{a',b'\in\sD'}\Rel(a',b')\sqcup\coprod_{a'',b''\in\sD''}
\Rel(a'',b'')$.

 Finally, let $\sN$ be the product of copies of the category $\sA$
indexed by the finite set $\coprod_{a'\in\sD',a''\in\sD''}\Rel(a',a'')$.
 Let $P\:\sK\rarrow\sN$ be the functor assigning to a functor
$A'\:\sD'\rarrow\sA$ the collection of objects whose component indexed
by a relation $r\in\Rel(a',a'')$ is the object $A'(a')\in\sA$.
 Similarly, let $Q\:\sL\rarrow\sN$ be the functor assigning to a functor
$A''\:\sD''\rarrow\sA$ the collection of objects whose component indexed
by a relation $r\in\Rel(a',a'')$ is the object $A''(a'')\in\sA$.
 Let $\phi\:P\Pi_\sK\rarrow Q\Pi_\sL$ be the natural transformation
assigning to a functor $A\in\sC=\Fun_k(\widetilde\sB,\sA)$ the morphism
in the category $\sN$ whose component indexed by a relation
$r\in\Rel(a',a'')$ is the morphism $A(r)\:A(a')\rarrow A(a'')$ for
all $a'\in\sD'$, \,$a''\in\sD''$.
 Let $\psi\:P\Pi_\sK\rarrow Q\Pi_\sL$ be the zero natural
transformation.
 Then the comma-category with relations $\sE$ is naturally equivalent
to the category of functors $\Fun_k(\sD,\sA)$.

 The induction assumption tells us that the categories $\sK$ and $\sL$
are $\kappa$\+accessible, and provides descriptions of their full
subcategories of $\kappa$\+presentable objects.
 Lemma~\ref{product-accessible} tells us that the categories $\sM$ and
$\sN$ are $\kappa$\+accessible, and provides descriptions of their
full subcategories of $\kappa$\+presentable objects.
 So Theorem~\ref{comma-category-with-equations-theorem} is applicable,
proving that the category $\sE=\Fun_k(\sD,\sA)$ is $\kappa$\+accessible,
and providing the desired description of its full subcategory of
$\kappa$\+presentable objects.
\end{proof}

\end{document}